\documentclass[sn-mathphys]{sn-jnl}

\usepackage{natbib}

\theoremstyle{thmstyleone}%
\newtheorem{theorem}{Theorem}
\newtheorem{lemma}[theorem]{Lemma}
\newtheorem{proposition}[theorem]{Proposition}
\newtheorem{corollary}[theorem]{Corollary}%

\theoremstyle{thmstyletwo}%
\newtheorem{example}{Example}%
\newtheorem{remark}{Remark}%

\theoremstyle{thmstylethree}%
\newtheorem{definition}{Definition}%
\newtheorem{condition}{Condition}%

\begin{document}

\title[Selection between distribution tail models]{Location- and scale-free procedures for distinguishing between distribution tail models}


\author{\fnm{Igor} \sur{Rodionov}}\email{igor.rodionov@essex.ac.uk}

\affil{\orgdiv{School of Mathematics, Statistics and Actuarial Science}, \orgname{University of Essex}, \orgaddress{\street{Wivenhoe Park}, \city{Colchester}, \postcode{CO4 3SQ}, \country{United Kingdom}}}

\abstract{We consider distinguishing between two distribution tail models when tails of one model are lighter (or heavier) than those of the other. 
Two procedures are proposed: one scale-free and one location- and scale-free, and their asymptotic properties are established. We show the advantage of using these procedures for distinguishing between certain tail models in comparison with the tests proposed in the literature by simulation and apply them to data on daily precipitation in Green Bay, US and Saentis, Switzerland.}



\maketitle

\section{Introduction}


Most methods used in the statistical analysis of distribution tails are provided by extreme value theory and involve the estimation of parameters of one of the general distribution tail models. Applications of this theory are very diverse, from finance and insurance to environment and biology. The peaks-over-threshold (POT) approach based on the Generalized Pareto distribution (GP) model is typically used for estimating distribution tails
, whereas the Gumbel block maximum method and the approach based on the extremal index are usually applied for estimating the maxima of time series
(see, e.g., \cite{dehaan} or \cite{beirlantteugels}). However, whereas financial and insurance data are usually heavy-tailed and are well approximated by the GP model, the tails of environmental and biological data may be lighter.

The GP model and generalized extreme value (GEV) distributions have been widely applied to fit extremes of environmental data, see, e.g., \cite{koutsoyiannis} or \cite{serinaldi}. However, \cite{maraniignaccolo} and \cite{miniussi} showed that fitting time series maxima using the Weibull model outperforms the classical Gumbel block maximum method on rainfall and flood data, respectively. Moreover, Weibull-type and log-Weibull-type tail models have shown their efficiency in estimating extremes of spectrometric data \cite{gardesgirard}, lifetimes of cancer patients \cite{wormss} and even financial time series \cite{malevergne}.
Surprisingly, there is almost no work on distinguishing between these tail models. The present article is devoted to this problem.

Let $\overline{F}(x) = 1 - F(x)$ denote the tail distribution function of a cdf $F,$ and let $x^{\ast}_{F} := \sup\{y: F(y)<1\}$ denote its right endpoint. Following \cite{resnick2}, we say that $\overline{F}$ and $\overline{G}$ are strictly tail-equivalent, if $x^\ast_{F} = x^\ast_{G} =: x^\ast$ and $\overline{F}(x)/\overline{G}(x) \to 1$ as $x\uparrow x^{\ast}$ (write $\overline{F}\sim\overline{G}$). We call {\it the (right) distribution tail} the equivalence class $T(\overline{F})$ of tail distribution functions, i.e., for all $\overline{G}$ the property $\overline{G}\sim\overline{F}$ is equivalent to $\overline{G}\in T(\overline{F}).$ 

Let $\textbf{X}_n = (X_1, \ldots, X_n)$ be i.i.d. random variables with cdf $F$ and with $x_F^\ast = +\infty.$ Below we propose a location- and scale-free model selection procedure for distinguishing between two tail models $A_0$ and $A_1$ using $\textbf{X}_n.$ Here $A_0$ and $A_1$ are families of distribution tails that are ``separable'' by some cdf $F_0$; in particular, all tails from $A_0$ are lighter (or heavier) than those from $A_1.$ We also propose a scale-free procedure for distinguishing $A_0$ from $A_1,$ that may be more powerful than the first. Both procedures are extensions of those proposed in \cite{R1}, which are neither location- nor scale-free. 

The article is organized as follows. In Section \ref{recent}, we discuss Weibull-type and log-Weibull-type models for distribution tails. 
In Section \ref{ccondition} we introduce the notion of C-separability, and then scale-free and location- and scale-free model selection procedures for distinguishing between C-separable classes of distribution tails are proposed in Sections \ref{scale-free} and \ref{location-free}, respectively, and their asymptotic properties are established. These procedures can be used for distinguishing between regularly-varying, Weibull-type and log-Weibull-type tails, though their applicability is not limited to this. The selection of the optimal threshold level for distinguishing between distribution tail models is discussed in Section \ref{kselection}. In Section \ref{simulationstudy}, the performance of our methods is examined on simulated and real data sets and compared with methods in the literature. All proofs are postponed to Section \ref{proofs}.


\section{Recent models for distribution tails} \label{recent}

The classical approach to handling extremes presupposes that the distribution function $F$ of observations belongs to the maximum domain of attraction of some extreme value distribution, written $F \in \mathcal{D}(EV_\gamma)$, i.e., for some sequences $a_n>0$ and $b_n,$
\begin{equation}\label{extremevaluetheorem}F^n(a_n x + b_n) \to EV_\gamma(x), \quad n\to\infty,\end{equation} where
\[EV_\gamma(x) = \left\{\begin{array}{cc} \exp(-(1+\gamma x)^{-1/\gamma}),\; 1+\gamma x>0, & \mbox{ if } \gamma\neq 0,\\
\exp(-\exp(-x)),\; x\in\mathbb{R}, & \text{ if }\gamma=0,\end{array}\right.\] is the generalized extreme value distribution and $\gamma$ is the extreme value index. If $F \in \mathcal{D}(EV_\gamma),$ then as $u\uparrow x^\ast_F,$ the distribution of excesses over a high threshold $u$
\[F_u(x) = \frac{F(x+u) - F(u)}{1 - F(u)}, \quad 0 \le x \le x_F^\ast - u,\]
can be approximated by the Generalized Pareto distribution $GP_\gamma(x/\sigma(u)),$ where
\[GP_\gamma(x) = \begin{cases} 1 - (1 + \gamma x)^{-1/\gamma},\;1 + \gamma x>0, \; x>0, & \text{ if } \gamma \neq 0,\\
1 - \exp(-x), \; x>0, & \text{ if } \gamma = 0,\end{cases}\] according to \cite{balkema} and \cite{pickands}.

The efficiency of the GP model to fit distribution tails belonging to the Gumbel MDA (that is, if $\gamma=0$) is sometimes unsatisfactory, so several models for tails within the Gumbel MDA have been proposed in last two decades. First, consider the model developed in \cite{devalk1} (see also the particular case of this model proposed in \cite{elmethni}). Following the description of this model in \cite{girard2020}, assume that the distribution tail of random variable $X$ (let $X\ge 1$ a.s. for simplicity) is given by
\begin{equation}1-F(x) = \exp(-V^{\leftarrow}(\ln x)),\ x\ge 1,\label{cees}\end{equation} where $V^{\leftarrow}(x) := \inf\{y: V (y) \ge x\}$ is the generalized inverse of $V(x) = \ln u_F(1/\exp(x))$ with $u_F,$ the tail
quantile function of $F,$ that is $u_F(t) = \overline{F}^{\leftarrow}(1/t).$ It is also supposed that there exists a positive function $a,$ such that for all $t>0$ and some real $\vartheta,$
\begin{equation}\label{ceesV}\lim_{x\to\infty} \frac{V(tx) - V(x)}{a(x)} = \frac{t^\vartheta - 1}{\theta},\end{equation} where for $\vartheta=0$ the right-hand side should be understood as $\ln t.$ Hence, $V$ is supposed to be of extended regular variation with index $\vartheta,$ see p. 371 in \cite{dehaan}.
If $\vartheta>0$ then $\vartheta$ governs the tail behavior of the log-Weibull-type distribution, i.e., $F$ is such that
\begin{equation}
\label{deflogw}\lim_{x\to\infty} \frac{\ln (1 - F(e^{t x}))}{\ln (1 - F(e^x))} = t^{1/\vartheta}, \quad t>0.\end{equation} Representatives of this class are the lognormal distribution and distributions with cdf $F(x) = 1 - \exp(-(\ln x)^{1/\vartheta}),$ for $x>1.$ If $0<\vartheta<1,$ then $F$ belongs to the Gumbel MDA and if $\vartheta = 1,$ then $F$ belongs to the Fr\'echet MDA (see Theorem 1 in \cite{girard2020}). The estimation of $\vartheta,$ extreme quantiles and probabilities of rare events for this model are considered in \cite{devalk1}, \cite{girard2020}, \cite{devalk2}, and \cite{devalk3}.

Another class of distributions from the Gumbel MDA is the class of Weibull-type distributions \cite{Girard}, \cite{Goegebeur}. We say that $F$ is of Weibull type if there exists $\theta>0$ such that
\begin{equation}\label{defw}\lim_{x\to\infty} \frac{\ln (1 - F(t x))}{\ln (1 - F(x))} = t^{1/\theta}, \quad t>0.\end{equation}
Then $\theta$ is called the Weibull tail index. Representatives of this class include the normal, exponential, gamma and Weibull distributions. The model \eqref{cees} 
does not provide good description for Weibull-type distributions, since in \eqref{ceesV} the parameter $\theta$ is equal to 0 for all distributions from this class as well as for distributions with lighter tails. The problems of estimation of the Weibull tail index and extreme quantiles for Weibull-type distributions are investigated by \cite{elmethni}, \cite{Girard}, \cite{Bro}, \cite{Beirlant}, and \cite{Gardes}, among others.

The literature on distinguishing between the Gumbel and Fr\'echet MDA (i.e., if $\gamma>0$) is quite extensive \cite{hueslerpeng}, \cite{nevesfraga}, but there is an almost complete lack of work on distinguishing between distribution tails within the Gumbel MDA. Apart from \cite{R1}, we mention \cite{Goegebeur}, where the goodness-of-fit test for Weibull-type behavior was proposed, \cite{R2}, dealing with distinguishing between Weibull-type distribution tails, and rather theoretical paper \cite{R3}, where the close hypotheses about distribution tails within the Gumbel MDA are considered. 
The tests proposed in these four articles are not invariant with respect to scale and location changes, though the test of \cite{Goegebeur} is scale-free. However, poor estimation or incorrect setting of the location and scale parameters can lead to serious errors in conclusions (see, e.g., Figure \ref{fig:weibull2}), which motivates developing location- and scale-free alternatives.

\section{Asymptotic results} \label{mainresults}

Hereinafter we assume that all considered distributions are continuous and that their right endpoints are infinite. Our methods can be extended to finite right endpoints, but this is beyond the scope of our paper. The model selection procedures introduced in Sections \ref{scale-free} and \ref{location-free} are generalizations of those of \cite{R1}, which are neither location- nor scale-free.

Hereafter for model selection procedures considered below, we call the choice of an
alternative class for distributions from a null class the type I error. Also, we
call the probability of selecting an alternative class if a distribution belongs to
it the power of model selection procedure.

\subsection{$C$-condition and $C$-separability. Examples.}\label{ccondition}

Denote the tail quantile function of $F$ by $u_F(t),$ that is $u_F(t) = \inf\{x: F(x) = 1 - 1/t\}.$ Let us introduce the conditions that we use to establish the asymptotic properties of the tests proposed in this work.
\begin{definition}
\label{def1} Two cdfs $H$ and $G$ are said to satisfy the $C$-condition
$C_\delta(H,G)$, if for some $\delta\ge 0$ there exists $t_0$ such that for all $t>t_0$ and $c>1,$
\begin{equation}
\frac{u_H\big(c^{1+\delta}t\big)}{u_G(ct)} \le \frac{u_H(t)}{u_G(t)}.
\label{cond}%
\end{equation}
\end{definition}

If $H$ and $G$ satisfy the $C_0$-condition
$C_0(H,G)$, then starting from some $t_0,$ the ratio $u_H(t)/u_G(t)$ does not increase.
Clearly, $C_{\delta}(H, G)$ is weaker than $C_{\delta^\prime}(H, G)$ for $\delta^\prime>\delta\ge 0.$

It is easy to see that $C_\delta(H,G)$ is invariant with respect to changes of the scale parameters of $G$ and $H.$ However, these conditions may be violated if $\sigma_1u_H(t) + a_1$ and $\sigma_2u_H(t) + a_2$ are substituted for $u_H(t)$ and $u_G(t),$ respectively, but this does not prevent us from proving the results below.

The condition $C_\delta(H,G)$ implies the following relation, which we use in proving the main results.

\begin{proposition}\label{P1}
Assume that cdfs $H$ and $G$ satisfy $C_\delta(H,G)$ for some $\delta\ge 0$ and $t_0.$ Then for all $t>t_0$ and $c>1,$
    \begin{equation}\frac{\big(1 - H(c u_H(t))\big)^{1-\varepsilon}}{(1 - H(u_H(t)))^{1-\varepsilon}} \le \frac{1 - G(c u_G(t))}{1 - G(u_G(t))}, \label{B'}
    \end{equation} with $\varepsilon = 1 - (1+\delta)^{-1}.$
\end{proposition}

Now we introduce the notion of $C$-separability.
\begin{definition} \label{def2}
Let $A_0$ and $A_1$ be two classes of distribution tails such that the tails belonging to $A_0$ are lighter than those from $A_1.$ We call $A_0$ and $A_1$ {\tt
$C$-separable} (from the right) if there exists a cdf $F_0$ such that for all $G\in A_0$ the condition $C_0(G, F_0)$ holds for some $t_0 = t_0(G),$ and for all $H\in A_1$
the condition $C_\delta(F_0, H)$ holds for some $\delta>0$ and $t_0 = t_0(H).$

If classes $A_0$ and $A_1$ are $C$-separable from the right by cdf $F_0,$ we call them $C(F_0)$-separable from the right.
\end{definition}
If the distribution tails in $A_0$ are heavier then those in $A_1,$ then we say that these two classes are $C(F_0)$-separable (from the left), if for some cdf $F_0$ and all $G\in A_0$ the condition $C_0(F_0, G)$ holds for some $t_0 = t_0(G),$ and for all $H\in A_1$ the condition $C_\delta(H, F_0)$ holds for some $\delta>0$ and $t_0 = t_0(H).$

$C$-separability is satisfied for a wide range of classes of distribution tails, as is now discussed.

\begin{example}\label{E1} Let $A_0$ and $A_1$ be the classes of Weibull-type and log-Weibull-type distribution tails, respectively (see Section \ref{recent}), and suppose that all such functions with tails in $A_0$ and $A_1$ are eventually differentiable. To distinguish between these two classes we recommend selecting the cdf
\begin{equation}F_0(x) = (1 -
\exp\{-\exp\{b(\ln x)^{1/2}\}\}) I(x>1)\label{w-lw}\end{equation} for some constant $b>0.$ The proof of $C(F_0)$-separability of these classes from the right under some technical conditions is given in Section \ref{proofs}.
The proof of their $C(F_0)$-separability from the left is similar. Finite-sample properties of procedures for distinguishing $A_0$ from $A_1$ are compared in Section \ref{sim1}.
\end{example}

\begin{example}\label{E2} Let $A_0$ be the class of log-Weibull-type distribution tails with $\vartheta<1,$ and let $A_1$ be the class of regularly-varying distribution tails. Recall that the Fr\'echet MDA consists of only the distributions with regularly-varying tails, and write $F\in \mathcal{D}(EV_\gamma),\ \gamma>0$ for such distribution functions. Assume again that all distribution functions with tails in the classes $A_0$ and $A_1$ are eventually differentiable. To distinguish between these two classes we recommend selecting the cdf
\begin{equation}\label{lw-rv} F_0(x) = 1 - \exp(-\exp(b\,\sqrt{\ln\ln x}) \ln x), \quad x>e,
\end{equation} for some constant $b>0.$
Indeed, the tails of log-Weibull-type distributions with $\vartheta<1$ are lighter than that of $F_0,$ and the condition $C_0(G, F_0)$ for $G\in A_0$ follows from \eqref{deflogw} and the properties of regularly-varying functions. On the other hand, the condition $C_\delta(F_0, H)$ for $H\in \mathcal{D}(EV_\gamma), \gamma>0,$ and $\delta>0$ follows from \eqref{lw-rv} and the definition of regular variation. We do not provide the detailed proofs of these facts, because they are technical and in many respects similar to the proofs of corresponding facts in Example 1. $C(F_0)$-separability from the right holds in this case as well.
Numerical comparisons of performance of tests for distinguishing between these two classes are provided in Section \ref{sim2}.
\end{example}

\begin{example}\label{E3} Let the class $A_0$ be such that $u_G(t) = \int_{1}^t f(x)/x\, dx$ for all $G\in A_0,$ where $f$ is positive and $f(x)\to \infty$ as $x\to\infty.$ It is clear that $G$ is heavy tailed. Next, let $A_1$ be the class of the tails of absolutely continuous Weibull-type distributions with $\theta<1$. Then the classes $A_0$ and $A_1$ are $C(F_0)$-separable from the left by $F_0,$ the cdf of the standard exponential distribution. We omit the proof of this fact, since it is similar to the proof of Example 1.
\end{example}

\begin{example}\label{E4} Consider a one-parameter family of distributions
 $F\in \{G_\theta, \theta\in\Theta\}$ such that for all $\theta_1,\theta_2\in \Theta,$ $\theta_1<\theta_2,$ the cdfs $G_{\theta_1}$ and $G_{\theta_2}$ satisfy the condition $C_\delta(G_{\theta_1}, G_{\theta_2})$ for $\delta>0.$
 Then it is easy to see that the classes $A_0 = \{G_\theta,\ \theta< \theta_0\}$ and $A_1 = \{G_\theta,\ \theta> \theta_0\}$ are $C(G_{\theta_0})$-separable from both the right and the left.
 \end{example}

\subsection{Scale-free model selection} \label{scale-free}

Assume that classes $A_0$ and $A_1$ are $C(F_0)$-separable from the right by some cdf $F_0$. In this section, we discuss the following procedure for distinguishing $A_0$ from $A_1:$
\begin{equation}\text{select } \begin{cases}A_0, & \text{ if } \tilde{R}_{k,n} \le 1+ \frac{u_{1-\alpha}}{\sqrt{k}},\\ A_1, & \text{ otherwise,} \end{cases}\label{test}\end{equation}
where $u_{1-\alpha}$ denotes the $(1-\alpha)-$quantile of the standard normal distribution,
\begin{equation}
\tilde{R}_{k,n} = \ln\frac{k}{n} -
\frac{1}{k}\sum_{i=n-k+1}^{n}\ln\big(1-F_0(u_0(n/k)X_{(i)}/X_{(n-k)})\big),
\label{rnk1} \end{equation} and $u_0(t) = u_{F_0}(t)$. Here $\alpha$ is the asymptotic upper bound on the probabilities of selecting $A_1$ for distribution tails in $A_0$ obtained when applying \eqref{test} to distinguish between  classes $A_0$ and $A_1$; it can be considered as analogous to a significance level.

The statistic $\tilde R_{k,n}$ is obtained by replacing the order statistics $\{X_{(i)}\}_{i=n-k}^n$ with their scale-free counterparts $\{u_0(n/k)X_{(i)}/X_{(n-k)}\}_{i=n-k}^n$ from the statistic \begin{equation}\label{initialrkn} 
R_{k,n} = \ln(1 - F_0(X_{(n-k)})) - \frac{1}{k} \sum_{i=n-k+1}^n \ln(1 - F_0(X_{(i)}))   
\end{equation} 
introduced in \cite{R1} and inspired by the Hill estimator \cite{hill} of the positive extreme value index $\gamma$. The statistic $R_{k,n}$ has fruitful distributional properties allowing to propose a procedure similar to \eqref{test}. Specifically, if $F = F_0,$ the distribution of $R_{k,n}$ is known, whereas if the right tail of $F$ is heavier or lighter than that of $F_0$ then $R_{k,n}$ tends in probability to $+\infty$ or $-\infty,$ respectively, as $k = k(n) \to \infty, k/n \to 0, n\to\infty.$ However, the result of applying this procedure to finite samples might depend strongly on changes in the scale parameters of $F$ and $F_0,$ which is a significant drawback of the method.

Unlike $R_{k,n},$ the distribution of
$\tilde{R}_{k,n}$ does not depend on the scale parameters of $F_0$ and $F,$ so the procedure \eqref{test} is scale-free. However, $\tilde{R}_{k,n}$ preserves distributional properties similar to those of $R_{k,n},$ see Theorem \ref{T1} and Theorem \ref{T2} below. If for some cdf $F_0$ classes $A_0$ and $A_1$ are $C(F_0)$-separable from the left, then 
we use the following procedure to distinguish $A_0$ from $A_1:$ 
\[\text{select } \begin{cases}A_0, & \text{ if } \tilde{R}_{k,n} \ge 1+ \frac{u_{\alpha}}{\sqrt{k}},\\ A_1, & \text{ otherwise,} \end{cases}\]
where $\alpha$ has the same meaning as above.

Now let us discuss the asymptotic properties of the procedure \eqref{test}. To formulate them, we need the following condition, which is classical for extreme value theory, see, e.g., formula (1.1.30) in \cite{dehaan}.

\begin{condition}\label{Cond1} We say that a cdf $F$ satisfies the von Mises condition if for some $\gamma\in \mathbb{R}$
\begin{eqnarray} \label{mises}
 \lim_{x \uparrow x^\ast_F} \frac{(1-F(x))F^{\prime\prime}(x)}{(F^\prime(x))^2} = -\gamma-1.
\end{eqnarray}
\end{condition}
It is known (see, e.g., Theorem 1.1.8 ibid.) that the von Mises condition is sufficient for $F$ to belong to $\mathcal{D}(EV_\gamma).$ Hereinafter we call an natural sequence $k = k(n)$ satisfying 
\begin{equation}\label{kn} k\to\infty, k/n\to 0 \mbox{ as } n\to\infty,
\end{equation}
an intermediate sequence, which is also classical for statistics of extremes.

The first result gives the behavior of the statistic $\tilde{R}_{k,n}$ if a cdf $F$ of a sample $X_1,\ldots, X_n$ coincides with $F_0.$

\begin{theorem} \label{T1} Let $X_1,\ldots, X_n$ be i.i.d. random variables with a cdf $F_0$ 
satisfying Condition \ref{Cond1}. 
If a sequence $k = k(n)$ is intermediate,
then
\[
\sqrt{k}(\tilde{R}_{k,n} - 1) \stackrel{d}{\rightarrow} N(0,
1), \qquad n\to\infty.
\]
\end{theorem}

Two further results together with Theorem \ref{T1} imply that the upper bound over the type I error probabilities of the procedure \eqref{test} for distribution tails from $A_0$ is at most $\alpha$ asymptotically, and the power of \eqref{test} tends to $1$ as $n\to\infty$.

\begin{theorem} \label{T2} Let $X_1,\ldots, X_n$ be i.i.d. random variables with a cdf $F_1$. Assume $F_0$ satisfies Condition \ref{Cond1} and that a sequence $k = k(n)$ is intermediate. If the condition $C_\delta(F_0, F_1)$ holds for some $\delta>0$, then
$$\sqrt{k}(\tilde{R}_{k,n} - 1) \xrightarrow{P} +\infty, \qquad n\to\infty,$$
and if $C_\delta(F_1, F_0)$ holds for some $\delta>0$, then
$$\sqrt{k}(\tilde{R}_{k,n} - 1) \xrightarrow{P} -\infty, \qquad n\to\infty.$$
\end{theorem}

\begin{corollary}\label{C1}
Under the assumptions of Theorem \ref{T2}, let $X_1^0,\ldots, X_n^0$ be i.i.d. random variables with cdf $F_0.$ Denote the values of the statistic \eqref{rnk1} built by samples $(X_1^0,\ldots, X_n^0)$ and $(X_1,\ldots, X_n)$ by $\tilde{R}_{k,n}^0$ and $\tilde{R}_{k,n}^1,$ respectively. If $C_0(F_0, F_1)$ holds, then for all $x>0$ eventually
\[P(\tilde{R}_{k,n}^0 >x) \le P(\tilde{R}_{k,n}^1 >x),
\]
but if $C_0(F_1, F_0)$ holds, then for all $x>0$ eventually
\[P(\tilde{R}_{k,n}^0 >x) \ge P(\tilde{R}_{k,n}^1 >x).
\]
\end{corollary}


\begin{remark}\label{R1} Although procedure \eqref{test} possesses the asymptotic properties mentioned above for all $F_0$ satisfying Condition \ref{Cond1} and such that classes
$A_0$ and $A_1$ are $C(F_0)$-separable, the choice of $F_0$ matters for finite $n$. However, the problem of optimal selection of $F_0$ is beyond the scope of this paper. This remark is related also to the procedure proposed in the next section. However, some practical recommendations can be given; see Section \ref{pracrec}.
\end{remark}

\begin{remark}\label{R2}
 It follows from the results of this section that if classes $A_0$ and $A_1$ are $C(F_0)$-separable 
and $F_0$ does not belong to $A_0$, then 
the type I error probabilities of \eqref{test} when applying it to distinguish $A_0$ from $A_1$ are asymptotically lower than $\alpha$. Thus, if $A_0$ has a natural ``left bound'' $\tilde F,$ i.e., for all $G\in A_0$ the condition $C_0(G, \tilde F)$ holds, but $C_\delta(G, \tilde F)$ does not necessary hold for any $\delta>0$, and then we recommend selecting $F_0 = \tilde F$ to maximize the power of the procedure.
\end{remark}

\subsection{Location- and scale-free model selection} \label{location-free}
The procedure proposed in the previous section and the tests mentioned at the end of Section \ref{recent} are not invariant with respect to the location parameter, which can restrict their application in some situations; see, e.g., Fig. \ref{fig:weibull2}.
 In this section, we propose a procedure for distinguishing between separable classes of distribution tails that is invariant with respect to both scale and location parameters. Let us consider the statistic
\begin{eqnarray}
\label{rnk2} &&\widehat{R}_{k,n} = \ln\frac{k}{n} - \\&&
\frac{1}{k}\sum_{i=n-k+1}^{n}\ln\left[1-F_0\Big(u_0(n/k) + \frac{X_{(i)} - X_{(n-k)}}{X_{(n-k)} - X_{(n-2k)}}\big(u_0(n/k) - u_0(n/(2k))\big)\Big)\right],
\nonumber \end{eqnarray}
which can be obtained from $R_{k,n}$ \eqref{initialrkn} by replacing the order statistics $\{X_{(i)}\}_{i=n-k}^n$ with their location- and scale-free counterparts
\[u_0(n/k) + \frac{X_{(i)} - X_{(n-k)}}{X_{(n-k)} - X_{(n-2k)}}\big(u_0(n/k) - u_0(n/(2k)), \quad i = n-k, \ldots, n.\]
The asymptotic properties of $\widehat{R}_{k,n}$ are stated below.

\begin{theorem}
\label{T3} Under the conditions of Theorem \ref{T1}, 
\[
\sqrt{k}(\widehat{R}_{k,n} - 1) \stackrel{d}{\rightarrow} N\Big(0,
1 + \frac{\gamma^2}{2(\gamma+1)^2(2^\gamma-1)^2}\Big), \qquad n\to\infty,
\]
where the ratio on the right-hand side should be understood as $1/(2(\ln 2)^2)$ for $\gamma=0.$
\end{theorem}

\begin{theorem} \label{T4} Assume the conditions of Theorem \ref{T2}. If the condition $C_\delta(F_0, F_1)$ holds for some $\delta>0$ then
$$\sqrt{k}(\widehat{R}_{k,n} - 1) \xrightarrow{P} +\infty,\qquad n\to\infty,$$ and if $C_\delta(F_1, F_0)$ holds for some $\delta>0$ then
$$\sqrt{k}(\widehat{R}_{k,n} - 1) \xrightarrow{P} -\infty,\qquad n\to\infty.$$
\end{theorem}

\begin{corollary}\label{C2}
Assume the conditions of Theorem \ref{T2}. Let $X_1^0,\ldots, X_n^0$ be i.i.d. random variables with the cdf $F_0.$ Denote the values of statistic \eqref{rnk2} built by samples $(X_1^0,\ldots, X_n^0)$ and $(X_1,\ldots, X_n)$ by $\widehat{R}_{k,n}^0$ and $\widehat{R}_{k,n}^1,$ respectively. If $C_0(F_0, F_1)$ holds then for all $x>0$ eventually
\[P(\widehat{R}_{k,n}^0 >x) \le P(\widehat{R}_{k,n}^1 >x),
\]
and if $C_0(F_1, F_0)$ holds then for all $x>0$ eventually
\[P(\widehat{R}_{k,n}^0 >x) \ge P(\widehat{R}_{k,n}^1 >x).
\]
\end{corollary}

Denote for convenience
\[\sigma^2(\gamma) = 1 + \frac{\gamma^2}{2(\gamma+1)^2(2^\gamma-1)^2}, \quad \gamma>0,\] and set $\sigma^2(0) = 1 + 1/(2(\ln 2)^2)$ as its limit as $\gamma\to 0.$ As in the previous section, consider two classes $A_0$ and $A_1$ of distribution tails such that tails in $A_0$ are lighter than those in $A_1,$ and assume that these classes are $C(F_0)$-separable from the right, where $F_0 \in \mathcal{D}(EV_\gamma).$ Let us propose the following procedure to distinguish  $A_0$ from $A_1:$
\begin{equation}\text{select} \begin{cases} A_0, & \text{ if } \widehat{R}_{k,n} \le 1+ \frac{\sigma(\gamma) u_{1-\alpha}}{\sqrt{k}}, \\
A_1, & \text{ otherwise.}\end{cases} \label{test-location}\end{equation}
This procedure is clearly location- and scale-free by the definition of $\widehat{R}_{k,n}.$ Moreover, Theorem \ref{T3}, Theorem \ref{T4}, and Corollary \ref{C2} imply that the upper bound over the type I error probabilities of procedure \eqref{test-location} is at most $\alpha$ asymptotically, and its power tends to $1$ as $n\to\infty$. The procedure for distinguishing $A_0$ from $A_1$ in case of their $C(F_0)$-separability from the left can be proposed similarly to the previous section.

\begin{remark} \label{R3} Assume that $A_0$ and $A_1$ are $C(F_0)$-separable from the right. Formally speaking, the procedures \eqref{test} and \eqref{test-location} distinguish not the classes $A_0$ and $A_1,$ but the wider classes $\hat A_0$ and $\hat A_1,$ where $\hat A_0$ consists of all (continuous) cdfs $G$ satisfying condition $C_0(G, F_0)$, and $\hat A_1$ consists of all (continuous) cdfs $H$ satisfying $C(F_0, H).$
\end{remark}

\begin{remark} Distributions belonging to $A_0$ or $A_1$ (apart from the cdf $F_0$ in case $F_0\in A_0$) do not have to satisfy the assumptions of the extreme value theorem \eqref{extremevaluetheorem}.
\end{remark}

\subsection{The choice of $k$} \label{kselection}

An important problem for practitioners is the optimal choice of $k.$ It should not be too small, since the more observations we use, the more accurate our conclusions are, but it should not be too large because, roughly speaking, the tail may ``end'' earlier and our conclusions about tail behavior in this case may turn out to be incorrect.

There are many papers devoted to the optimal selection of $k$ in more general context. For instance, we refer to \cite{Dani} for a review of such methods for tail index estimation, see also Section 5.4 in \cite{Gomes2}, and to \cite{Marrod} for those for extremal index estimation. These methods are most often based on either empirical principles, like choosing the optimal value of estimator from some interval of stability of its values, or on optimization of some metrics (e.g., MSE), often using the bootstrap.

However, the methods of choosing $k$ proposed in the framework of statistics of extremes are unsuitable when distinguishing between distribution tail models. Indeed, statistics of such procedures and tests quite often have a steady increasing or decreasing trend as $k$ increases, so the stability interval method cannot be used, and almost never tends to a certain finite value, implying that the method of optimizing some metric in the form that used for parameter estimation cannot be applied either. To the best of our knowledge, so far in papers devoted to testing hypotheses about distribution
tails this question has been omitted. Moreover, the most comprehensive review to date of methods for testing hypotheses about distribution tails \cite{hueslerpeng} suggests that the optimal selection of $k$ is an open problem.

Nevertheless, judging by our observations, some regularities can be distinguished. As a rule, a test for a hypothesis about the distribution tail (or a model selection procedure in our case) can be represented as
\[ \mbox{ if } S(k) > u, \mbox{ then reject } H_0 \text{ (select } A_1),\] where the statistic $S(k)$ is built from the $k$ largest order statistics of a sample, and $u$ is a constant independent of $k.$ For instance, all the procedures considered in this article can be represented in such a way. Moreover, it is often possible to ensure that the limit distribution of $S(k)$ does not depend on $k$ under the null hypothesis (or for some ``boundary'' distributions, like in this article, see Theorems \ref{T1} and \ref{T3}), where $k$ is an intermediate sequence, i.e., it satisfies \eqref{kn}. Let us consider such a case. Then three types of $S(k)$ behavior are most common as $k$ increases, and the type of this behavior does not change on different samples from the same distribution. Below we provide recommendations for testing the null hypothesis (or selecting the distribution tail model) for these three types of behavior:
\begin{enumerate} \item $S(k)$ has a well expressed increasing trend, and its values exceed $u$ almost immediately. In this case, $H_0$ should be rejected (the class $A_1$ should be selected);
\item $S(k)$ has a well expressed decreasing trend, and its values become smaller than $u$ almost immediately. In this case, $H_0$ should not be rejected (the class $A_0$ should be selected);
\item the values of $S(k)$ oscillate around $u$ for some interval of small values of $k.$ As a rule, there is no stability of the $S(k)$ values in this case. If there is no stable interval of these values below or above level $u,$ then no decision is possible. 
\end{enumerate}

Of course, these recommendations are empirical, and the problem of selection of $k$ while testing hypotheses about distribution tails deserves a separate study. Possible solutions to this problem may be the use of multiple hypothesis testing or sequential analysis.

\section{Simulation study} \label{simulationstudy}

In this section, we examine the efficiency of the procedures introduced by the article in comparison with some tests proposed in the literature. Before this, we provide some practical recommendations on selecting the parameters of $F_0$ for better performance of the procedures. We also apply our methods to data from the Daily
Global Historical Climatology Network (GHCNd) \cite{menne}. 


The role of the parameter $\alpha$ below is the following: it denotes the significance level for all tests considered in this section and corresponds to the asymptotic upper
bound over the empirical type I error probabilities of the model selection procedures \eqref{test} and \eqref{test-location}. We set its value equal to $0.05.$ The number of replicates used to calculate empirical characteristics below is always equal to $2000.$

\subsection{Practical recommendations on selecting $F_0$}\label{pracrec}

When discussing the selection of $F_0$ in Remarks \ref{R1} and \ref{R2}, we noted that we do not provide a procedure for the optimal choice of $F_0$ (and, consequently, its parameters). Nevertheless, some practical recommendations can be given.

First note that the statistic used in the location- and scale-free procedure does not depend on the scale and location parameters of $F_0,$ and the statistic used in the scale-free procedure does not depend on the scale parameter of $F_0,$ so these parameters need not be chosen.

Next, assume that classes $A_0$ and $A_1$ are $C(F_0)$-separable from the right. If $A_0$ has a natural ``left bound'' (see Remark \ref{R2} for definition and, e.g., Examples \ref{E3}, \ref{E4}), then $F_0$ should be chosen equal to this bound. In this case the scale and location parameters need not be chosen as well. 

If $A_0$ has no natural ``left bound'', like in Examples \ref{E1} and \ref{E2}, it is enough to vary only one parameter to derive a good quality procedure. We suggest choosing some ``basic'' distribution in $A_0$ quite close to $A_1$ in terms of tail heaviness (for example, in Section \ref{sim1} we choose the Weibull distribution with cdf $F(x) = 1 - \exp(-x^{a}), x>0,$ and $a=0.5$) and selecting the shape parameter of $F_0$ in such a way that the average empirical type I error probability would be close to $\alpha$. For this, the preliminary simulation study should be carried out, see Section \ref{f0selection}.


Finally, we note that using the location- and scale-free procedure is preferable when distinguishing between distributions from the Gumbel MDA, in which the location parameter can strongly affect the tail behavior for small samples (see Fig. \ref{fig:weibull2}).
Moreover, given fixed classes $A_0$ and $A_1$, for different sample sizes the same separating cdf $F_0$ can be used.

\subsection{Impact of $F_0$ choice on procedure performance}\label{f0selection} In this section, we discuss the impact of the choice of separating cdf $F_0$ on the performance of the procedures proposed by this article. 

First, let us give the details of the preliminary simulation procedure mentioned in Section \ref{pracrec} whose purpose is to select the optimal value of the shape parameter given a fixed choice of $F_0$ family. We focus on the performance of the procedure \eqref{test-location} for distinguishing Weibull-type from log-Weibull-type tails, which is discussed in detail in Section \ref{sim1}. We select $F_0$ given by \eqref{w-lw} and show how the analogs of the type I error probabilities and the power of this procedure change depending on the choice of the parameter $b.$ Two distributions are selected:
\begin{itemize}
    \item the Weibull distribution $\mathcal{W}(0.5,1)$ with cdf
    \begin{equation}F(x) = 1 - \exp(-x^{0.5}), x>0,\label{w051}\end{equation}
    whose tail is of Weibull type;
    \item the standard lognormal distribution $\mathcal{LN}(0,1)$ with density
    \[p(x) = \frac{1}{\sqrt{2\pi} x} \exp\big(-(\ln x)^2/2\big),\]
    whose tail is of log-Weibull type.
\end{itemize} 
Recall that it is the distribution \eqref{w051} that we propose as "basic" for distinguishing between Weibull-type and log-Weibull-type tails according to the recommendations given in Section \ref{pracrec}.

For each distribution, we generated $5000$ samples of size $n = 2500$ and
computed the rejection rates of the procedure \eqref{test-location} for $k$ from $5$ up to $500$ in steps of $5.$ In Figure \ref{f01} we plot the empirical I type error probabilities (left panel) and empirical power (right panel) of this procedure for $b = 2.5, 3, 3.5, 4,$ and $4.5,$ the shape parameter of $F_0$ given by \eqref{w-lw}.

\begin{figure}[ht]
\begin{minipage}[h]{0.48\linewidth}
\center{\includegraphics[width=1\linewidth]{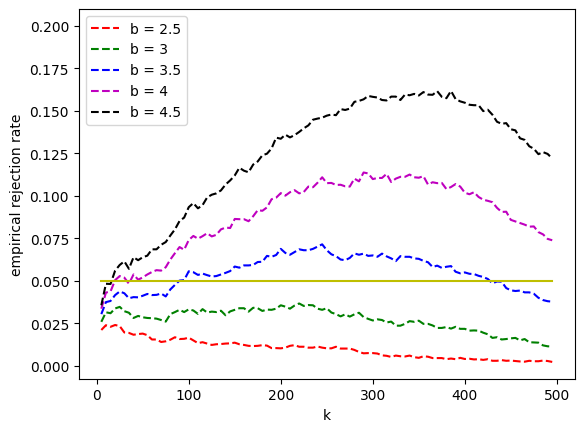} 
\\ {\small $\mathcal{W}(0.5,1)$}}
\end{minipage}
\hfill
\begin{minipage}[h]{0.48\linewidth}
\center{\includegraphics[width=1\linewidth]{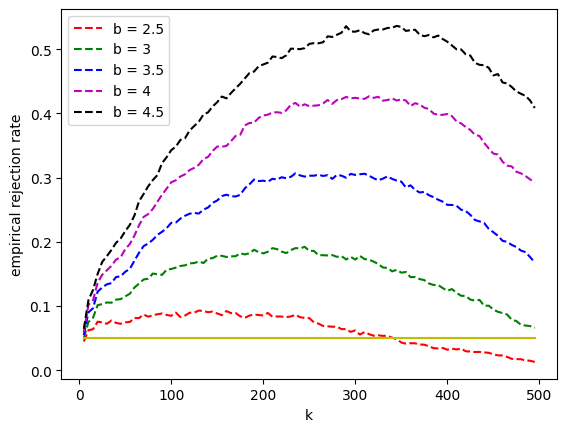} 
\\ {\small $\mathcal{LN}(0,1)$}}
\end{minipage}
\caption{{\footnotesize Empirical type I error probabilities (left) and empirical power (right) of the procedure \eqref{test-location} for distinguishing Weibull-type from log-Weibull type tails with $F_0$ given by \eqref{w-lw} and 5 different values of its shape parameter $b,$ the level $\alpha = 0.05$ is indicated by a yellow solid line, $n=2500.$}}
\label{f01}
\end{figure}
According to Figure \ref{f01} and the recommendations given in Section \ref{pracrec}, the best choice for the parameter $b$ would be $3.5$ because the average empirical type I error probability of the procedure \eqref{test-location} on the "basic" distribution $\mathcal{W}(0.5,1)$ is closer to $0.05$ than those for other values of $b.$ Of course, the choice of $b$ strongly affects the power of the procedure, see the right panel of Figure \ref{f01}: the smaller $b$ the smaller the empirical power.

It is also worth to analyse how much the performance of the procedure changes if we vary the form of the separating cdf $F_0$. Let us consider distinguishing regularly-varying from Weibull-type distribution tails, in this case it is reasonable to select $F_0$ of log-Weibull type. Indeed, every log-Weibull-type distribution is separating for these two classes. So, we select four log-Weibull-type distributions as candidates for $F_0:$
\begin{itemize}
\item the log-Weibull distribution $\mathcal{LW}(\lambda, 1)$ with cdf \[F(x) = 1 - \exp( - (\ln x)^\lambda), \quad x>1,\; \lambda>0,\] and $\lambda = 1.5, 2$ and $3;$
\item the standard lognormal distribution $\mathcal{LN}(0,1);$ 
\end{itemize}
and analyse the performance of the procedure \eqref{test-location} on two distributions, where the first is regularly-varying and the second is of Weibull type:
\begin{itemize}
\item the generalized Pareto distribution $\mathcal{GP}(0.25, 1)$ with cdf \[F(x) = 1 - (1 + 0.25 x)^{-4}, \quad x>0;\] 
\item the standard exponential distribution $Exp(1).$
\end{itemize}
The $\mathcal{GP}(0.25, 1)$ distribution can be considered as "basic" in this example.
For each distribution, we generate $5000$ samples of size $n = 2500,$ compute the rejection rates of the procedure \eqref{test-location} with four candidate $F_0$ listed above, and plot these rates in Figure \ref{f02}.

\begin{figure}[ht]
\begin{minipage}[h]{0.48\linewidth}
\center{\includegraphics[width=1\linewidth]{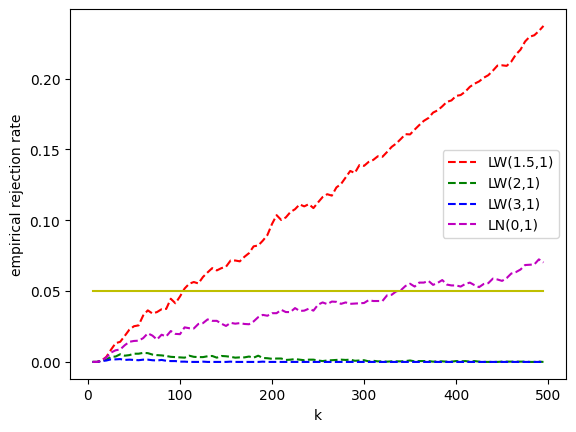} 
\\ {\small $\mathcal{GP}(0.25, 1)$}}
\end{minipage}
\hfill
\begin{minipage}[h]{0.48\linewidth}
\center{\includegraphics[width=1\linewidth]{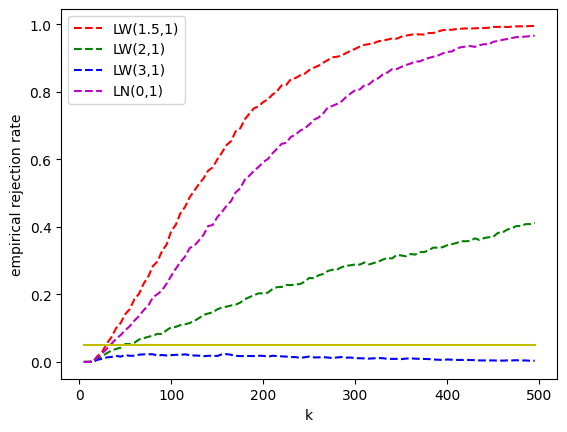} 
\\ {\small $Exp(1)$}}
\end{minipage}
\caption{{\footnotesize Empirical type I error probabilities (left) and empirical power (right) of the procedure \eqref{test-location} for distinguishing regularly-varying from Weibull-type tails with four $F_0$ listed above, the level $\alpha = 0.05$ is indicated by a yellow solid line, $n=2500.$}}
\label{f02}
\end{figure}

Despite $\mathcal{LW}(1.5,1)$ is the closest among considered distributions to the class of distributions with regularly-varying tails in terms of tail heaviness, it is not suitable to be a separating cdf for distinguishing between regularly-varying and Weibull-type tails because of high type I error probability. On the other hand, selecting $\mathcal{LW}(3,1)$ as $F_0$ for this problem has no sense due to poor performance of the procedure on the $Exp(1)$ distribution. Due to the recommendations given in Section \ref{pracrec}, it is reasonable to select the $\mathcal{LN}(0,1)$ distribution as $F_0$ for this problem (or consider for this purpose the $\mathcal{LN}(0, b)$ distribution with the shape parameter $b$ selected by the procedure described above in this section). It is worth noting the rather different performance of the procedures with $F_0 = \mathcal{LN}(0,1)$ and $F_0 = \mathcal{LW}(2,1)$ despite the similar behavior of their tails. This effect can be explained by the different shapes of these distributions which influences the different finite sample behavior of the procedure. 

\subsection{Distinguishing Weibull-type from log-Weibull-type di\-stri\-bu\-tion tails} \label{sim1}
 As mentioned in Section \ref{recent}, to the best of our knowledge, only one existing publication concerns distinguishing Weibull-type from log-Weibull-type tails, namely \cite{Goegebeur}. They introduced the Jackson- and Lewis-type goodness-of-fit tests for Weibull-type tail behavior.

We compare the procedures \eqref{test} and \eqref{test-location} adapted for distinguishing between the classes $\mathcal{W}$ and $\mathcal{LW},$ i.e., the classes of Weibull-type and log-Weibull-type distribution tails, respectively, with the Jackson- and Lewis-type tests from \cite{Goegebeur}. According to Remark \ref{R3} we can consider the distributions with tails heavier than those of $\mathcal{LW}$, in particular, regularly-varying distributions, to be included in $\mathcal{LW}$. As a separating cdf for procedures \eqref{test} and \eqref{test-location}, we select $F_0$ given by \eqref{w-lw} with $b$ equal to $1.8$ and $3.5,$ respectively, according to the procedure described in Section \ref{f0selection}. We examine the performance of the procedures and tests listed above on the following set of distributions, which mostly coincides with those used by \cite{Goegebeur}. First, we list Weibull-type distributions:

 \begin{itemize}
 \item The Weibull distribution $\mathcal{W}(a, \lambda)$
 \[F(x) = 1 - \exp(-\lambda x^a), \quad x>0,\;  a, \lambda>0,\]
 with $(a, \lambda) = (0.5, 1).$
 \item The standard normal distribution $\mathcal{N}(0,1).$
 \item The gamma distribution $\Gamma(a, \lambda)$ with density
 \[p(x) = \frac{\lambda^a x^{a-1}}{\Gamma(a)}e^{-\lambda x}, \quad x>0,\; \alpha, \lambda>0,\] with $(a, \lambda) = (0.25,1),\ (4,1).$
 \item The modified standard exponential distribution $\mathcal{ME}(1),$ defined as the distribution of the random variable $Y = X \ln X,$ where $X\sim Exp(1).$
 \item The extended Weibull model $\mathcal{EW}(a, \beta)$
  \[F(x) = 1 - r(x)e^{-x^a},\quad  x>0,\]
  where $a>0$ and $r(x)$ is regularly varying at infinity with index $\beta \in \mathbb{R}.$ We set $r(x) = (x+1)^{-1}$ and $a = 2.$
 \end{itemize}
 Now we list the log-Weibull-type distributions and distributions with heavier tails which we use in this section:
\begin{itemize}
 \item The lognormal distribution $\mathcal{LN}(\mu, \sigma)$ with density \[p(x) = \frac{1}{\sqrt{2\pi\sigma^2}x} \exp\left(-\frac{(\ln x - \mu)^2}{2\sigma^2}\right),\quad x>0,\, \mu\in \mathbb{R},\, \sigma^2>0,\] with $\mu=0$ and $\sigma^2=1.$
\item The log-Weibull distribution $\mathcal{LW}(\lambda, c)$ with cdf \[F(x) = 1 - \exp( - (\ln (x/c))^\lambda), \quad x>c,\; \lambda, c>0,\] and $(\lambda,c) = (1.5,1).$
\item The generalized Pareto distribution $\mathcal{GP}(\gamma, \sigma)$ with cdf \[F(x) = 1 - (1 + \gamma x/\sigma)^{-1/\gamma}, \quad x>0,\, \gamma, \sigma>0,\] with the cases $(\gamma, \sigma) = (0.25, 1), (0.5, 1), (1, 1),$ denoted by $P_1,$ $P_2$ and $P_3,$ respectively.
\item The t-distribution $\mathcal{T}(n)$ with $n=3$ degrees of freedom.
 \end{itemize}
For each distribution, we generated samples of size $n = 2500$ and
computed the rejection rates of all procedures and tests for $k$ from $5$ up to $500$ in steps of $5.$

\begin{figure}[ht]
\begin{minipage}[h]{0.32\linewidth}
\center{\includegraphics[width=1\linewidth]{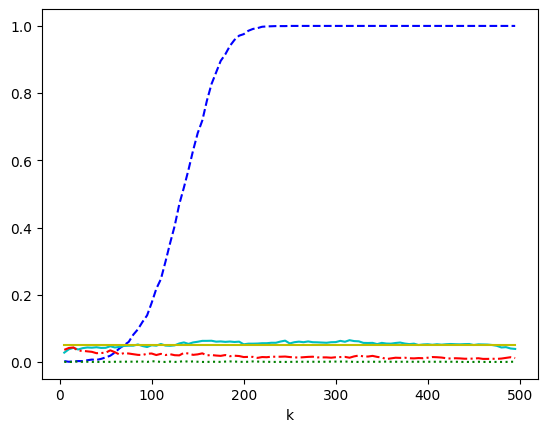}\\  
{\small $\mathcal{W}(0.5,1)$}}
\end{minipage}
\hfill
\begin{minipage}[h]{0.32\linewidth}
\center{\includegraphics[width=1\linewidth]{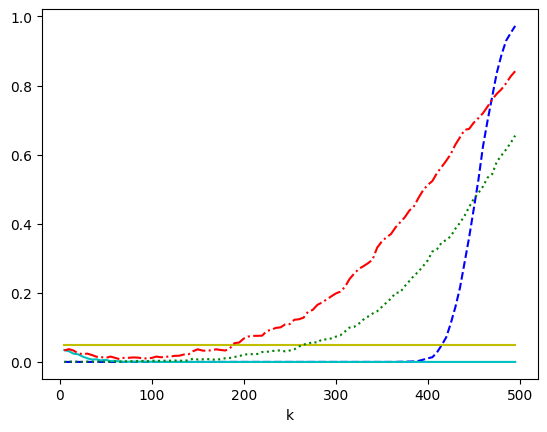}\\ 
{\small $\mathcal{N}(0,1)$}}
\end{minipage}
\hfill
\begin{minipage}[h]{0.32\linewidth}
\center{\includegraphics[width=1\linewidth]{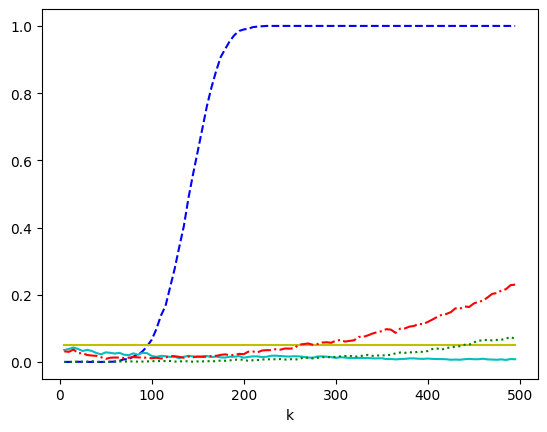}\\ 
 {\small $\Gamma(0.25, 1)$}}
\end{minipage}

\bigskip

\begin{minipage}[h]{0.32\linewidth}
\center{\includegraphics[width=1\linewidth]{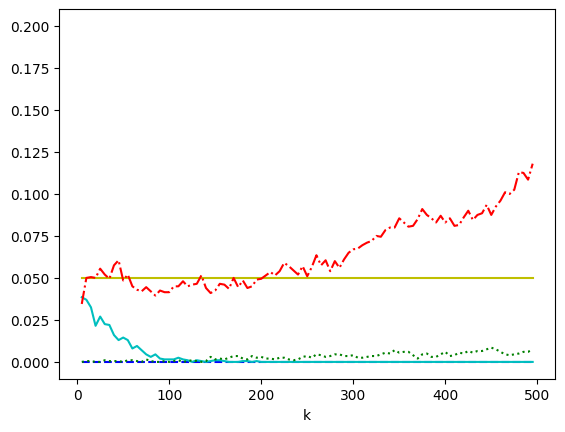} 
\\
{\small $\Gamma(4,1)$}}
\end{minipage}
\hfill
\begin{minipage}[h]{0.32\linewidth}
\center{\includegraphics[width=1\linewidth]{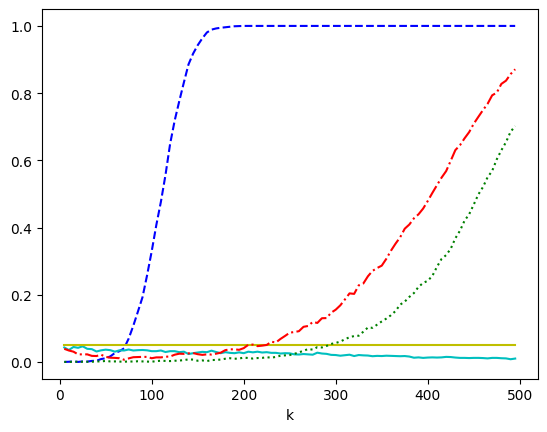} 
\\ {\small $\mathcal{ME}(1)$}}
\end{minipage}
\hfill
\begin{minipage}[h]{0.32\linewidth}
\center{\includegraphics[width=1\linewidth]{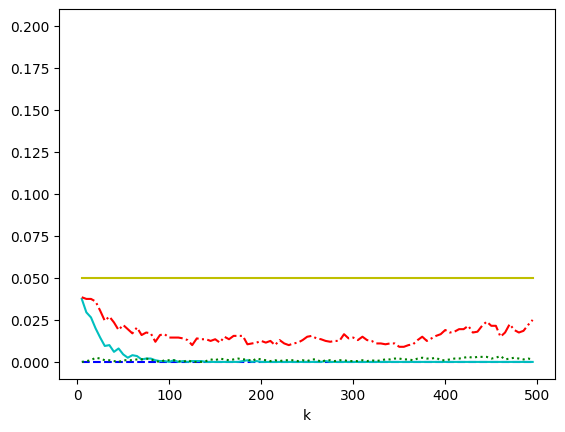} 
\\
{\small $\mathcal{EW}(2, -1)$}}
\end{minipage}
\caption{{\footnotesize Empirical type I error probabilities of the procedures \eqref{test} (blue dashed line), \eqref{test-location} (cyan solid line), Jackson-type (green dotted line) and Lewis-type (red dash-dotted line)  tests for various distributions, the level $\alpha = 0.05$ is indicated by a yellow solid line, $n=2500.$}}
\label{fig:weibull}
\end{figure}

\begin{figure}[ht]
\begin{minipage}[h]{0.32\linewidth}
\center{\includegraphics[width=1\linewidth]{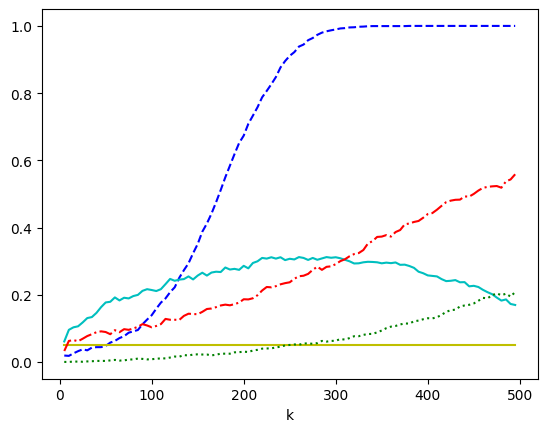} 
\\ {\small $\mathcal{LN}(0,1)$}}
\end{minipage}
\hfill
\begin{minipage}[h]{0.32\linewidth}
\center{\includegraphics[width=1\linewidth]{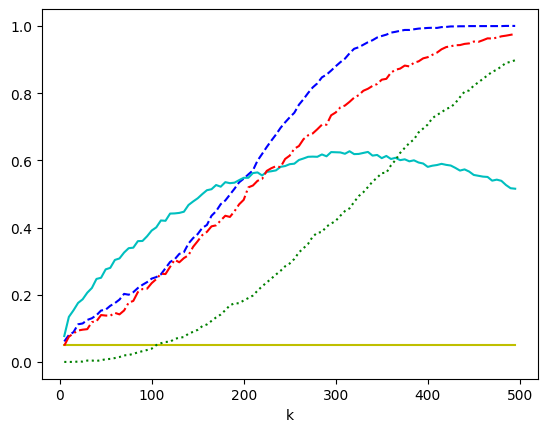} 
\\ {\small $\mathcal{LW}(1.5,1)$}}
\end{minipage}
\hfill
\begin{minipage}[h]{0.32\linewidth}
\center{\includegraphics[width=1\linewidth]{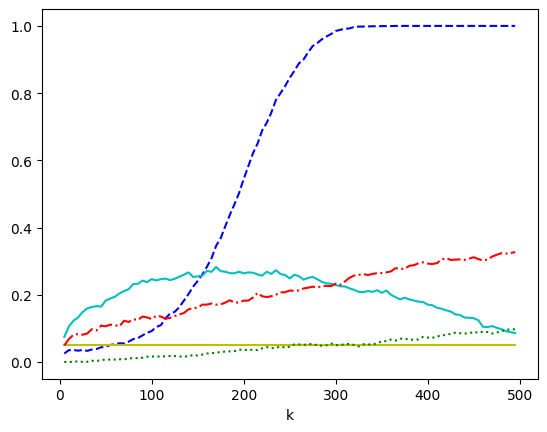} 
\\ {\small $\mathcal{GP}(0.25,1)$}}
\end{minipage}

\bigskip

\begin{minipage}[h]{0.32\linewidth}
\center{\includegraphics[width=1\linewidth]{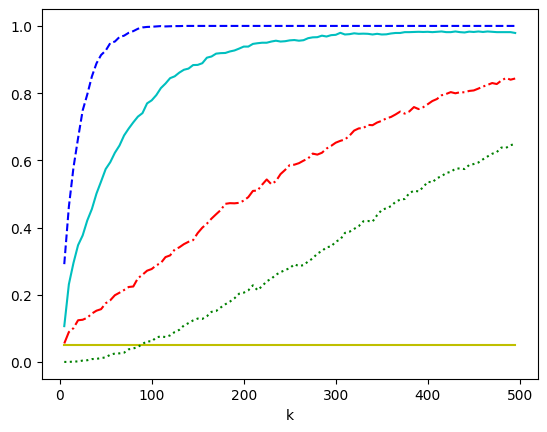} 
\\
{\small $\mathcal{GP}(0.5,1)$}}
\end{minipage}
\hfill
\begin{minipage}[h]{0.32\linewidth}
\center{\includegraphics[width=1\linewidth]{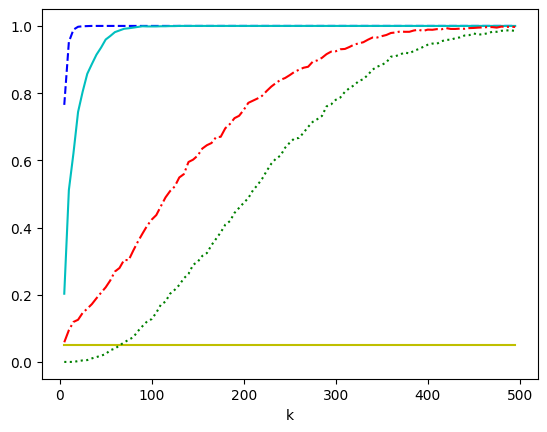} 
\\ {\small $\mathcal{GP}(1,1)$}}
\end{minipage}
\hfill
\begin{minipage}[h]{0.32\linewidth}
\center{\includegraphics[width=1\linewidth]{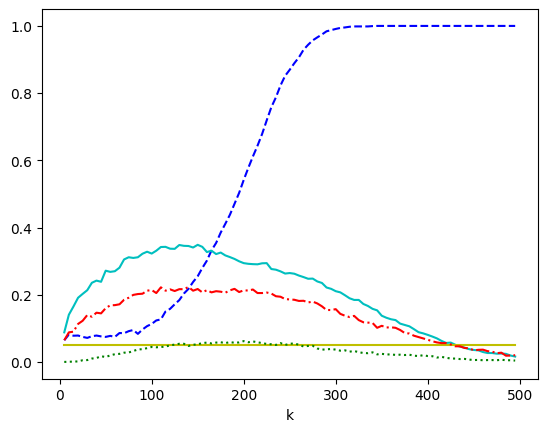} 
\\
{\small $\mathcal{T}(3)$}}
\end{minipage}
\caption{{\footnotesize Empirical power of the procedures \eqref{test} (blue dashed line), \eqref{test-location} (cyan solid line), Jackson-type (green dotted line) and Lewis-type (red dash-dotted line) tests for distributions from $\mathcal{W},$ the level $\alpha = 0.05$ is indicated by a yellow solid line, $n=2500$.}}
\label{fig:logweibull}
\end{figure}

In Figure \ref{fig:weibull} for the six distributions from $\mathcal{W}$ we plot the empirical type I error probabilities of the scale-free procedure \eqref{test}, location- and scale-free procedure \eqref{test-location}, Jackson- and Lewis-type tests, proposed by \cite{Goegebeur}. 
Since $\mathcal{W}$ is the null class, we expect to see empirical type I error probabilities below level $\alpha$ for distributions from $\mathcal{W};$ but only procedure \eqref{test-location} has these probabilities below $0.05$ (for $\mathcal{W}(0.5,1),$ close to $0.05$) for all $k\le 500$ and all distributions from $\mathcal{W}.$ However, the Jackson- and Lewis-type tests provide reasonable empirical type I error probabilities as long as $k$ is not too high, for instance, if $k\le 200.$ 
We focus on this interval of $k$ values. 

In Figure \ref{fig:logweibull} for six distributions from $\mathcal{LW}$ we plot empirical power of the procedures and tests mentioned above. Procedure \eqref{test-location} is more powerful than the Jackson- and Lewis-type tests and procedure \eqref{test} is more powerful than the Jackson-type test on distributions from the class $\mathcal{LW}$ when $k\le 200.$

Analyzing Fig. \ref{fig:weibull}-\ref{fig:logweibull},
 one may conclude that procedure \eqref{test-location}
  is always preferable to procedure \eqref{test}.
  The reason to use \eqref{test}
   instead of \eqref{test-location},
   in particular, can be a small sample; see Fig. \ref{fig:logweibull2}. It can also be reasonable to use \eqref{test}
   for distinguishing between quite heavy (in particular, regularly-varying) distribution tails. Fig. \ref{fig:logweibull2} shows empirical power of all tests and procedures considered in this section for distributions $\mathcal{LN}(0,1),$ $\mathcal{GP}(0.25, 1)$ and $\mathcal{T}(3)$ and $n=1000.$  The empirical power of \eqref{test}
   is substantially larger than the empirical power of other tests/procedures for moderate $k$.

\begin{figure}[ht]
\begin{minipage}[h]{0.32\linewidth}
\center{\includegraphics[width=1\linewidth]{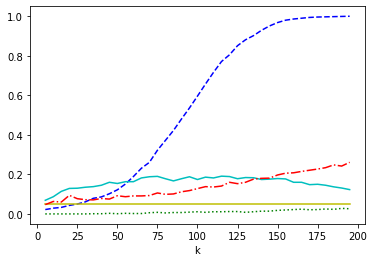} \\ $\mathcal{LN}(0,1)$}
\end{minipage}
\hfill
\begin{minipage}[h]{0.32\linewidth}
\center{\includegraphics[width=1\linewidth]{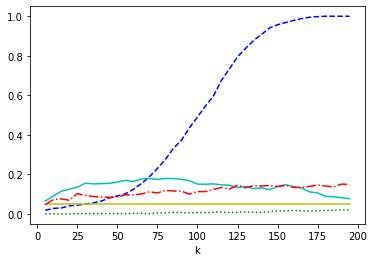} \\ $\mathcal{GP}(0.25,1)$}
\end{minipage}
\hfill
\begin{minipage}[h]{0.32\linewidth}
\center{\includegraphics[width=1\linewidth]{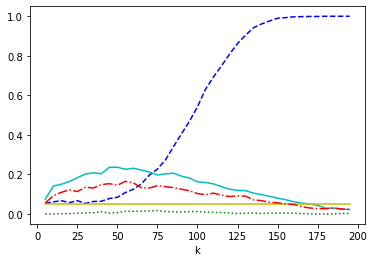} \\ $\mathcal{T}(3)$}
\end{minipage}
\caption{{\footnotesize Empirical power of the procedures \eqref{test}
(blue dashed line), \eqref{test-location}
(cyan solid line), the Jackson-type (green dotted line) and Lewis-type (red dash-dotted line) tests for distributions from $\mathcal{LW}$, the level $\alpha = 0.05$ is indicated by a yellow solid line, $n=1000.$}}
\label{fig:logweibull2}
\end{figure}
However, the Jackson- and Lewis-type tests and the procedure \eqref{test}
 are quite sensitive to the location parameter. In Fig. \ref{fig:weibull2}, we provide the empirical type I error probabilities of all the tests and procedures considered in this section for the exponential density $p(x) = \exp(-(x-a)) I(x>a)$ and values $a = -1.5, 0, 1.5$ of the location parameter, with $n=2500.$ For $a=0,$ the empirical type I error probabilities of all the tests and procedures are below $\alpha$ for all $k\le 500$ ($k\le 400$ for the test \eqref{test})
  whereas for $a = -1.5$ the empirical type I error probabilities of all the tests/procedures except \eqref{test-location}
   increase substantially. Moreover, for $a=1.5$ these probabilities for Lewis-type test are higher $0.05.$ Since even small changes of the location parameter can cause large changes in the properties of the Jackson- and Lewis-type tests as well as in \eqref{test},
   they should be applied with great caution in practice when information about the location parameter is lacking. Moreover, these tests cannot be applied to samples containing only negative observations, an additional difficulty when using them. Summarizing the above, we can conclude that the use of the procedure \eqref{test-location} as compared to the other three procedures considered above is more beneficial for distinguishing between the classes $\mathcal{W}$ and $\mathcal{LW}.$
\begin{figure}[ht]
\begin{minipage}[h]{0.32\linewidth}
\center{\includegraphics[width=1\linewidth]{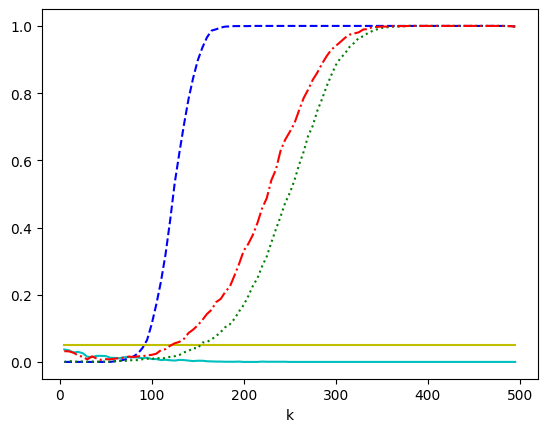} 
\\ $Exp(1,-1)$}
\end{minipage}
\hfill
\begin{minipage}[h]{0.32\linewidth}
\center{\includegraphics[width=1\linewidth]{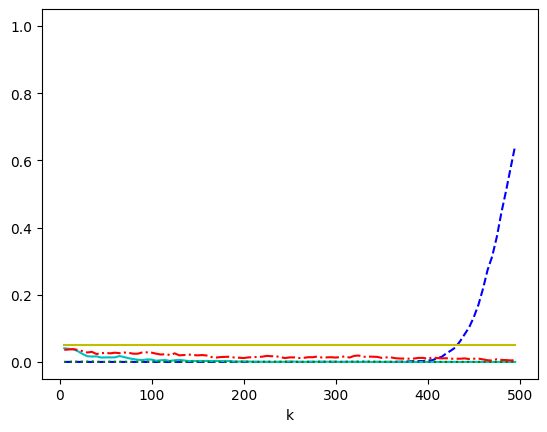} 
\\ $Exp(1,0)$}
\end{minipage}
\hfill
\begin{minipage}[h]{0.32\linewidth}
\center{\includegraphics[width=1\linewidth]{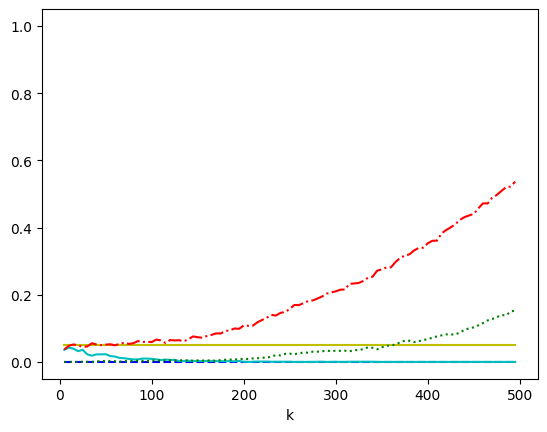} 
\\ $Exp(1,1)$}
\end{minipage}
\caption{{\footnotesize Empirical power of the procedures \eqref{test}
(blue dashed line), \eqref{test-location}
 (cyan solid line), the Jackson-type (green dotted line) and Lewis-type (red dash-dotted line) tests for exponential distributions with various values of the location parameter, the level $\alpha = 0.05$ is indicated by a yellow solid line, $n=2500.$}}
\label{fig:weibull2}
\end{figure}

\begin{remark}
The procedure \eqref{test-location} seems to be quite conservative in the
majority of situations (however, the same can be said about the Jackson-type test), the empirical type I error probabilities being very far from the
nominal level $0.05.$ This phenomenon can be explained by the analogy with the uniformly most powerful test for testing $H_0 :\theta \le \theta_0$ versus $H_1: \theta> \theta_0$ if a parametric family of distributions $\mathcal{P} = \{P_\theta, \theta\in \Theta \subset \mathbb{R}\}$ possesses the monotone likelihood ratio property. This test is in fact conservative for every distribution from the null hypothesis except $P_{\theta_0}$ which acts as a "boundary" for the null hypothesis, or as a "separating distribution" (compare with $F_0$ in our setting) for classes $\mathcal{P}_0 = \{P_\theta, \theta<\theta_0\}$ and $\mathcal{P}_1 = \{P_\theta, \theta>\theta_0\}.$ The $C_0$- 
 and $C$-conditions are somewhat analogs of the monotone likelihood ratio property for the procedures \eqref{test} and \eqref{test-location}, which are eventually conservative according to Corollary \ref{C1} and Corollary \ref{C2}, respectively.
\end{remark}


\subsection{Distinguishing log-Weibull-type from regularly-varying di\-stri\-bution tails} \label{sim2}

The range of tests that can be used for the problem of this section is quite wide. Here we consider the procedures \eqref{test}
 and \eqref{test-location}
 adapted for distinguishing between the classes $\mathcal{LW}$ and $\mathcal{RV},$ where $\mathcal{RV}$ is the class of distributions with regularly-varying tails, which coincides with the Fr\'echet MDA $\mathcal{D}(EV_\gamma)$ for $\gamma>0.$ We compare their efficiency with that of some methods for testing the hypothesis $\tilde H_0: F\in \mathcal{D}(EV_0),$ i.e., the hypothesis of belonging to the Gumbel MDA, against the alternative $H_1: F\in \mathcal{D}(EV_\gamma),$ $\gamma>0;$ see \cite{nevesfraga} for an overview of such methods. Namely, we use
\begin{enumerate}
\item the test of \cite{Hasofer-Wang}, see also \cite{Alves-Neves}, based on the Shapiro--Wilk type statistic
\[W_n(k) = \frac{k\big(\frac{1}{k}\sum_{i=0}^{k-1}X_{(n-i)} - X_{(n-k+1)}\big)^2}{(k-1)\sum_{j=0}^{k-1}\big(\frac{1}{k}\sum_{i=0}^{k-1}X_{(n-i)} - X_{(n-j)}\big)^2};\]
\item the test proposed in \cite{Alves-Neves} and based on the Greenwood-type statistic $G_n(k),$ a modification of the previous test;
\item the ratio test proposed in \cite{Picek} and based on the statistic
\[R_n(k) = \frac{X_{(n)} - X_{(n-k)}}{\frac{1}{k}\sum_{i=0}^{k-1}X_{(n-i)} - X_{(n-k)}};\]
\item and the likelihood ratio test, see, e.g., \cite{Picek}.
\end{enumerate}
Numerical comparisons of these tests can be found in \cite{Alves-Neves} and \cite{Picek}, see also \cite{nevesfraga}.

As a separating cdf for \eqref{test}
 and \eqref{test-location}
  we select $F_0$ given by \eqref{lw-rv}
   with $b$ equal to $0.6$ and $1.1,$ respectively. We compare the efficiency of the tests and procedures listed above on the following distributions (many of which were introduced in the previous section) belonging to the Gumbel MDA:

\begin{itemize}
\item the standard exponential distribution $Exp(1);$
\item the Weibull distribution $\mathcal{W}(0.5,1);$
\item the gamma distribution $\Gamma(0.25, 1);$
\item the modified standard exponential distribution $\mathcal{ME}(1);$
\item the standard lognormal distribution $\mathcal{LN}(0,1);$
\item the log-Weibull distribution $\mathcal{LW}(\lambda, c)$ with $(\lambda, c) = (1.5, 1), (3, 1);$
\end{itemize}
and on others belonging to the Fr\'echet MDA:
\begin{itemize}
\item the generalized Pareto distribution $\mathcal{GP}(\alpha, \sigma)$ with $(\gamma, \sigma) = (1/3, 1),$\\ $(0.5, 1), (1, 1);$ respectively.
\item the standard Cauchy distribution $\mathcal{C}(1);$
\item the t-distribution with 3 degrees of freedom $\mathcal{T}(3);$
\item the Burr distribution $Burr(c,d)$ with cdf
\[F(x) = 1 - \frac{1}{(1+ x^{-c})^d}, \quad x>0,\] with $(c,d) = (2, 2),\, (3,2).$
\end{itemize}

For each distribution we generated samples of size $n = 2500$ and
computed the empirical rejection rates of all tests and procedures for $k$ from $5$ up to $500.$
The empirical properties of the Hasofer--Wang, Greenwood-type and likelihood ratio tests are quite similar, so we do not show the results for the last two.

\begin{figure}[ht]
\begin{minipage}[h]{0.32\linewidth}
\center{\includegraphics[width=1\linewidth]{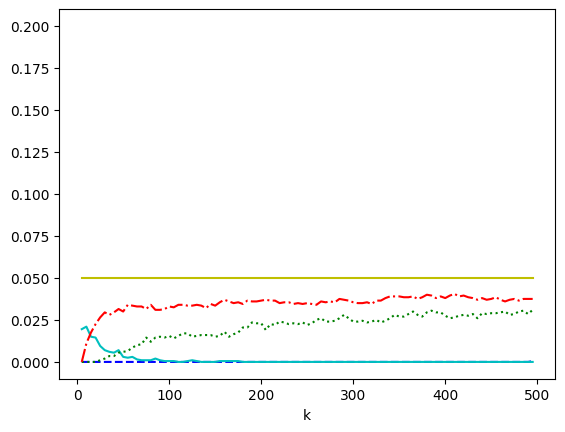} 
\\
{\small $Exp(1)$}}
\end{minipage}
\hfill
\begin{minipage}[h]{0.32\linewidth}
\center{\includegraphics[width=1\linewidth]{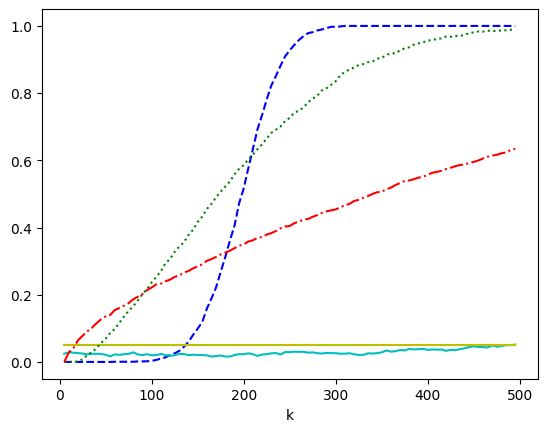} 
\\ {\small $\mathcal{W}(0.5,1)$}}
\end{minipage}
\hfill
\begin{minipage}[h]{0.32\linewidth}
\center{\includegraphics[width=1\linewidth]{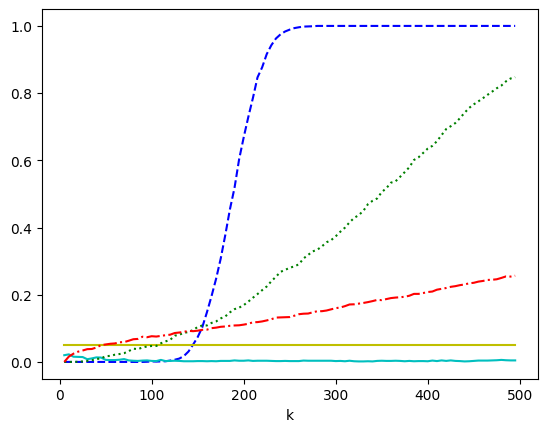} 
\\ {\small $\Gamma(0.25, 1)$}}
\end{minipage}

\bigskip

\begin{minipage}[h]{0.24\linewidth}
\center{\includegraphics[width=1\linewidth]{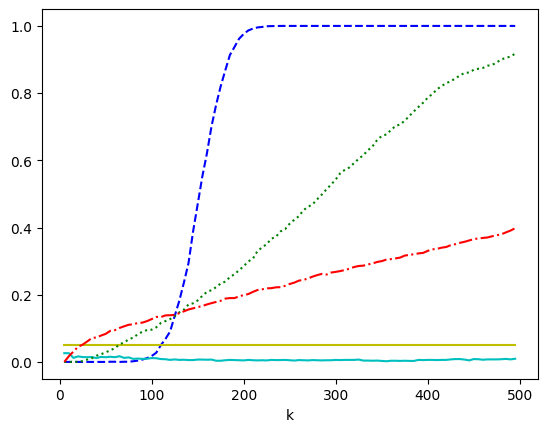} 
\\
{\small $\mathcal{ME}(1)$}}
\end{minipage}
\hfill
\begin{minipage}[h]{0.24\linewidth}
\center{\includegraphics[width=1\linewidth]{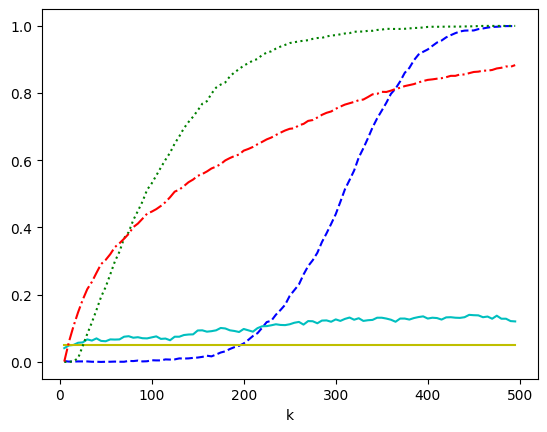} 
\\ {\small $\mathcal{LN}(0,1)$}}
\end{minipage}
\hfill
\begin{minipage}[h]{0.24\linewidth}
\center{\includegraphics[width=1\linewidth]{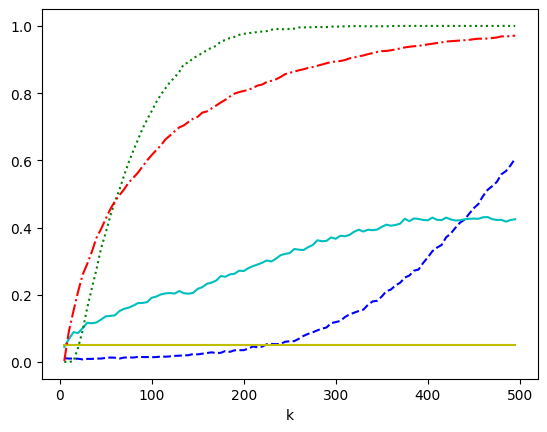} 
\\
{\small $\mathcal{LW}(1.5, 1)$}}
\end{minipage}
\hfill
\begin{minipage}[h]{0.24\linewidth}
\center{\includegraphics[width=1\linewidth]{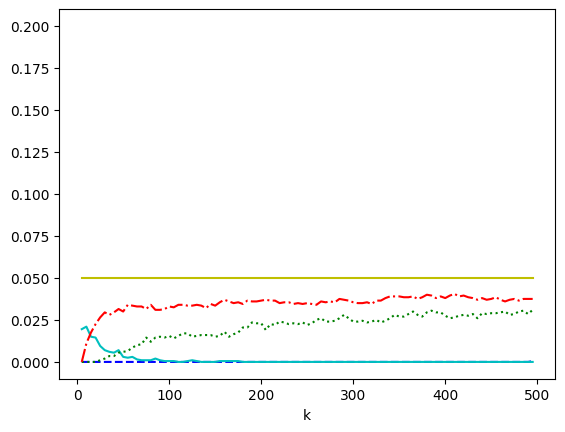} 
\\
{\small $\mathcal{LW}(3, 1)$}}
\end{minipage}
\caption{{\footnotesize Empirical type I error probabilities of the procedures \eqref{test}
 (blue dashed line), \eqref{test-location}
  (cyan solid line), the Hasofer--Wang (green dotted line) and ratio (red dash-dotted line) for various distributions, the level $\alpha = 0.05$ is indicated by a yellow solid line, $n=2500.$}}
\label{fig:lwrv1}
\end{figure}

\begin{figure}[ht]
\begin{minipage}[h]{0.32\linewidth}
\center{\includegraphics[width=1\linewidth]{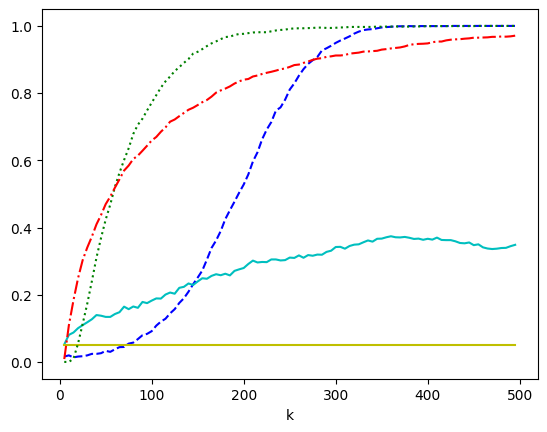} 
\\
{\small $\mathcal{GP}(1/3,1)$}}
\end{minipage}
\hfill
\begin{minipage}[h]{0.32\linewidth}
\center{\includegraphics[width=1\linewidth]{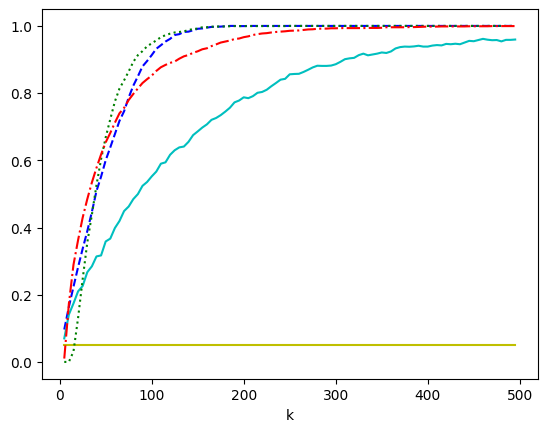} 
\\ {\small $\mathcal{GP}(0.5,1)$}}
\end{minipage}
\hfill
\begin{minipage}[h]{0.32\linewidth}
\center{\includegraphics[width=1\linewidth]{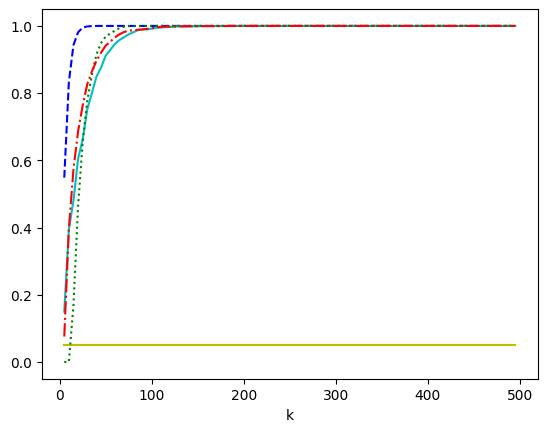} 
\\ {\small $\mathcal{GP}(1,1)$}}
\end{minipage}

\bigskip

\begin{minipage}[h]{0.24\linewidth}
\center{\includegraphics[width=1\linewidth]{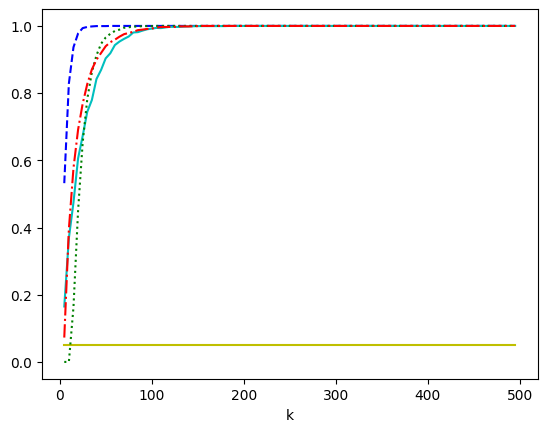} 
\\
{\small $\mathcal{C}(1)$}}
\end{minipage}
\hfill
\begin{minipage}[h]{0.24\linewidth}
\center{\includegraphics[width=1\linewidth]{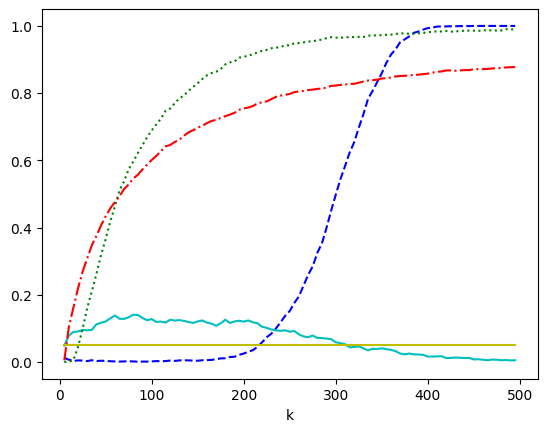} 
\\ {\small $\mathcal{T}(3)$}}
\end{minipage}
\hfill
\begin{minipage}[h]{0.24\linewidth}
\center{\includegraphics[width=1\linewidth]{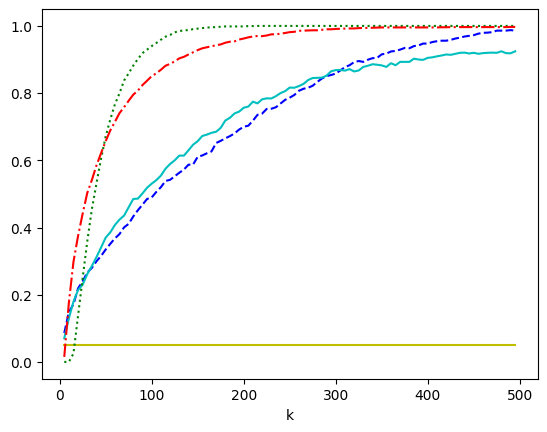} 
\\
{\small $Burr(2, 2)$}}
\end{minipage}
\hfill
\begin{minipage}[h]{0.24\linewidth}
\center{\includegraphics[width=1\linewidth]{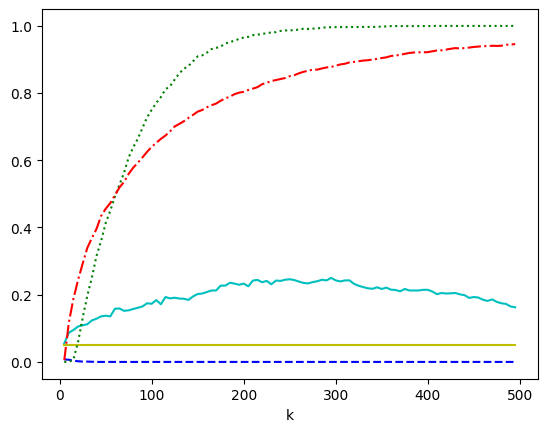} 
\\
{\small $Burr(3, 2)$}}
\end{minipage}
\caption{{\footnotesize Empirical power of the procedures \eqref{test}
 (blue dashed line), \eqref{test-location}
  (cyan solid line), the Hasofer--Wang (green dotted line) and ratio (red dash-dotted line) tests for various distributions, the significance level $\alpha = 0.05$ is indicated by a yellow solid line, $n=2500.$}}
\label{fig:lwrv2}
\end{figure}

In Figs. \ref{fig:lwrv1}–\ref{fig:lwrv2} we plot the empirical type I error probabilities for seven distributions from the Gumbel MDA and empirical power for seven distributions from the Fr\'echet MDA of the scale-free procedure \eqref{test}, location- and scale-free procedure \eqref{test-location}, Hasofer--Wang and ratio tests.

 For all the distributions from the Gumbel MDA, except for $Exp(1)$ and $\mathcal{LW}(3,1),$ the empirical type I error probabilities of the Hasofer--Wang and ratio tests are greater than $\alpha$ for almost all $k$ values. On the other hand, the empirical type I error probabilities of procedure \eqref{test-location}
exceed $\alpha$ only for $\mathcal{LN}(0,1)$ and $\mathcal{LW}(1.5,1).$ However, even for these two distributions, these probabilities are much less than those of the Hasofer--Wang and ratio tests. 
Procedure \eqref{test}
  is conservative for all $k<100$ on all the  distributions considered from the class $\mathcal{LW}$. But one should not forget that this procedure is the only one in our comparison that is not location-free. On the other hand, all the tests and procedures considered show quite good efficiency on the distributions from the class $\mathcal{RV},$ except poor performance of \eqref{test} and \eqref{test-location} on the distributions $\mathcal{T}(3)$ and $Burr(3,2).$

Summarizing the results of this section, we can conclude that the procedure \eqref{test-location}
 outperforms other tests/procedures analyzed in this section in terms of its empirical properties on the distributions from class $\mathcal{LW}$ but might have less power on those of $\mathcal{RV},$ like on $\mathcal{GP}(1/3,1),$ $\mathcal{T}(3),$ and $Burr(3,2).$ However, the high type I error probabilities of Hasofer--Wang and ratio tests do not allow us to consider them adequate for testing the hypothesis $H_0: F\in \mathcal{LW}$ and even $\tilde H_0: F\in \mathcal{D}(EV_0).$
 This means that tests/procedures for distinguishing between $\mathcal{D}(EV_0)$ and $\mathcal{D}(EV_\gamma),$ $\gamma>0,$ based on tail heaviness, might be more efficient than tests based on the properties of the GEV and GP models.

\subsection{Real data example
} \label{realdata}

In this section, the Jackson- and Lewis-type tests from \cite{Goegebeur} and the procedure \eqref{test-location} for distinguishing $\mathcal{W}$ from $\mathcal{LW}$ (see Section \ref{sim1}) are compared on daily precipitation data collected in Green Bay, US and Jena, Germany; the total number of observations in both datasets is $3630.$ The data was taken from the Daily Global Historical Climatology Network (GHCNd), an integrated database of daily climate summaries from land surface stations across the globe, \cite{menne}. The two datasets cover the period
1901--2021, but for stationarity reasons, only days in April were considered, leading to $1331$ and $1693$ non-zero observations in Green Bay and Jena, respectively. According to \cite{serinaldi} (see paragraph [38] there) which analyzes the GHCNd data as we do, there is no evidence of monotonic trend in the data. Moreover, to check for possible clustering of extremes, we estimate the extremal index of these data using the intervals estimator \cite{ferrosegers}. To apply it, we consider a set of thresholds equal to the empirical quantiles of levels from 0.9 to 0.99. For both datasets and all selected thresholds, the values of the intervals estimator lie above the 0.9 level, which confirms that the extremes of both datasets can be considered as asymptotically independent and the application of the Jackson- and Lewis-type tests as well as the procedure \eqref{test-location} to these data is correct.

We focus on the 150 largest observations of these samples. The p-values of the Jackson- and Lewis-type tests and analogs of p-values of the procedure \eqref{test-location} (calculated as $p = 1 - \Phi\big(\sqrt{k}(\widehat{R}_{k,n} - 1)/\sigma(0)\big)$, where $\Phi$ is the standard normal cdf) are shown in Figure \ref{fig:prcp} for $k = 10, \ldots, 150.$ Their quite unstable behavior is because there are many identical observations (the total number of unique non-zero observations is 127 in Green Bay and 172 in Jena; this is explained by the fact that the amount of precipitation in the GHCNd database is given with an accuracy of up to one millimeter).

\begin{figure}
[ht]\begin{minipage}[h]{0.48\linewidth}
\center{\includegraphics[width=1\linewidth]{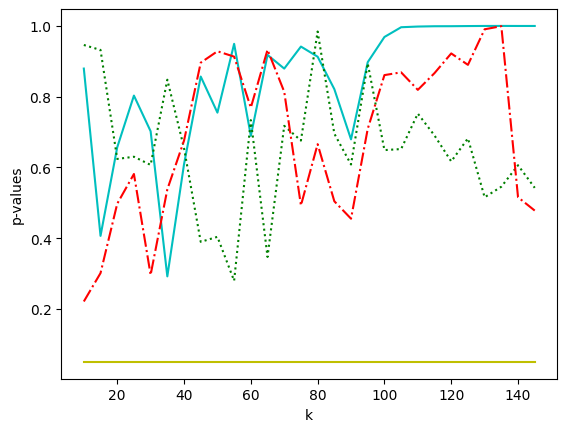}\\{\small Green Bay, US}}\end{minipage}
\hfill
\begin{minipage}[h]{0.48\linewidth}
\center{\includegraphics[width=1\linewidth]{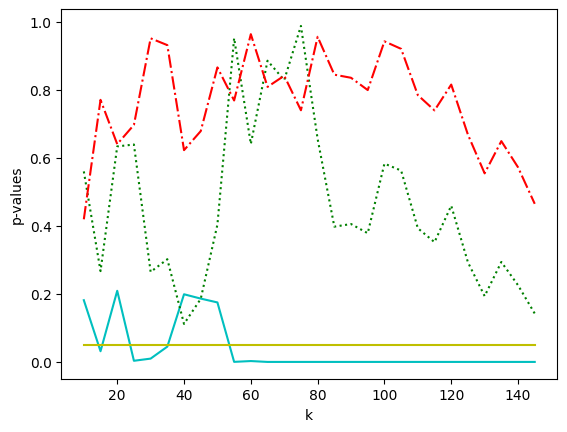}\\{\small Jena, Germany}}
\end{minipage}
\caption{{\footnotesize P-values of the procedure \eqref{test-location}
 (cyan solid line), the Jackson-type (green dotted line) and Lewis-type (red dash-dotted line) tests for daily precipitation data in Green Bay, US (left panel) and Jena, Germany (right panel); the level $\alpha = 0.05$ is indicated by a yellow solid line, $k=10, \ldots, 150.$}}
\label{fig:prcp}
\end{figure}

The left panel of Figure \ref{fig:prcp} suggests that the class $\mathcal{W}$ should be selected as an appropriate tail model for data collected in Green Bay, though traditionally extremes of precipitation data are modelled by extreme value distributions with positive shape parameter, see \cite{koutsoyiannis}. Indeed, p-values of the Jackson- and Lewis-type tests and its analogs of the procedure \eqref{test-location} are higher than the level $\alpha = 0.05$ for all $k$ values shown. 

This conclusion is confirmed by the quantile-quantile plots in Figure \ref{qqplots}. The left panel is a Hill plot, or Pareto quantile plot (\cite{beirlantteugels}, p.102), showing the fit of the largest sample observations by the Pareto distribution (and as a consequence by any regularly-varying distribution), and containing points
\[\left\{\left(\ln\big(\frac{n+1}{j}\big), \ln(X_{(n-j+1)}) \right); \quad j = 1, \ldots, 150\right\}.\] The right panel is a Weibull plot \cite{beirlantbladt}, since it is exactly a quantile-quantile plot for high quantiles of the Weibull distribution $\mathcal{W}(a, \lambda),$ see Section \ref{sim1} for definition. It is defined by points
\[\left\{\left(\ln\Big(\ln\big(\frac{n+1}{j}\big)\Big), \ln(X_{(n-j+1)}) \right); \quad j = 1, \ldots, 150\right\}\] and is inspired by approximation (6) in \cite{girard0}. The central panel is a log-Weibull plot also introduced in \cite{beirlantbladt}, a quantile-quantile
plot for high quantiles of the log-Weibull distribution $\mathcal{LW}(\lambda, c),$ containing points
\[\left\{\left(\ln\Big(\ln\big(\frac{n+1}{j}\big)\Big), \ln\big(\ln(X_{(n-j+1)})) \right); \quad j = 1, \ldots, 150\right\}\]
The quite good fit shown by the Weibull plot can be interpreted as an empirical validation of the Weibull-type behavior of the tail, whereas the Hill and log-Weibull plots are clearly concave.

\begin{figure}[ht]
\begin{minipage}[h]{0.32\linewidth}
\center{\includegraphics[width=1\linewidth]{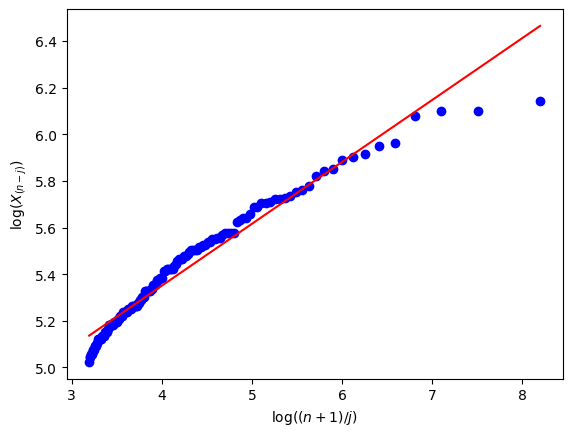}}
\end{minipage}
\hfill
\begin{minipage}[h]{0.32\linewidth}
\center{\includegraphics[width=1\linewidth]{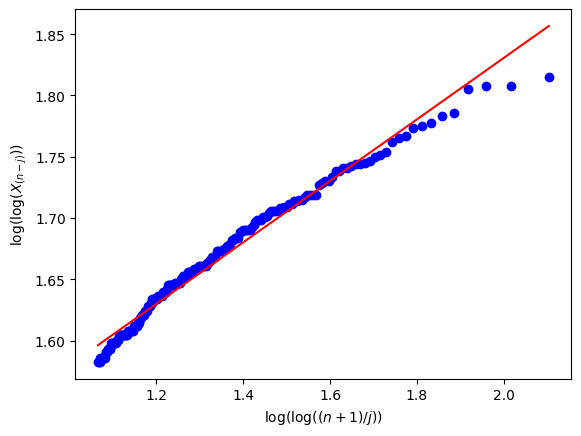}}
\end{minipage}
\hfill
\begin{minipage}[h]{0.32\linewidth}
\center{\includegraphics[width=1\linewidth]{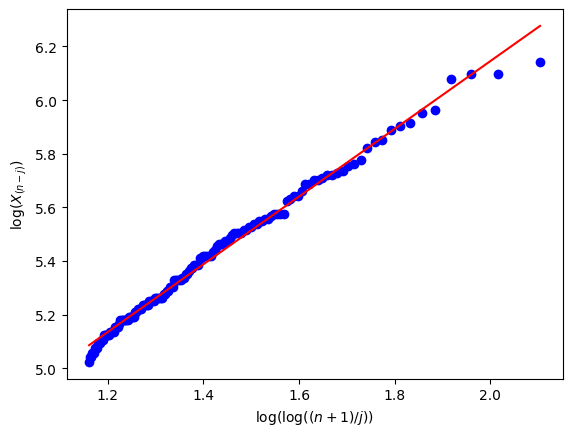}}
\end{minipage}
\caption{{\footnotesize Hill (left), log-Weibull (center) and Weibull (right) plots of the largest $k=150$ daily precipitation observations in Green Bay, US, in April during 1901--2021, $n=1331,$ the red solid lines correspond to the best linear fits to these data.}}
\label{qqplots}
\end{figure}

The right panel of Figure \ref{fig:prcp} is not so univocal as its left one. The Jackson- and Lewis-type tests suggest that the class $\mathcal{W}$ should be selected as a tail model for data collected in Jena. However, analogs of p-values of the procedure \eqref{test-location} occur both below and above the level $\alpha = 0.05;$ moreover, there is a stable interval of $\widehat{R}_{k,n}$ values below $\alpha$ for $k$ from 60 to 150, which suggests that the class $\mathcal{LW}$ (or, maybe, a model with heavier tails) should be selected as a tail model for this data according to Section \ref{kselection}. 

The quantile-quantile plots in Figure \ref{saentis} also do not allow us to draw a clear conclusion about which of the three models is the best. Whereas the Hill plot looks a bit concave, the log-Weibull and Weibull plots show a good fit, that is both the log-Weibull and Weibull tail models look to be well suited to describe the precipitation data collected in Jena. 

\begin{figure}[ht]
\begin{minipage}[h]{0.32\linewidth}
\center{\includegraphics[width=1\linewidth]{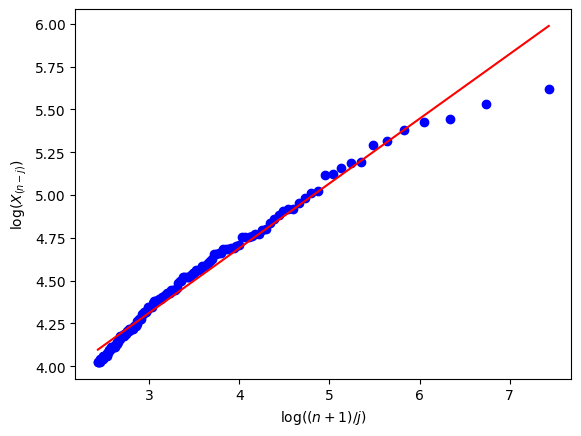}}
\end{minipage}
\hfill
\begin{minipage}[h]{0.32\linewidth}
\center{\includegraphics[width=1\linewidth]{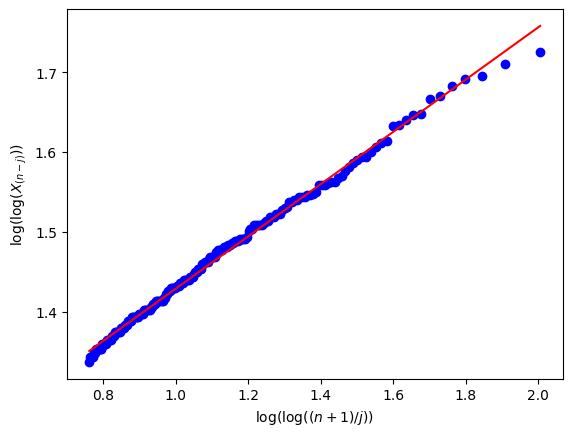}}
\end{minipage}
\hfill
\begin{minipage}[h]{0.32\linewidth}
\center{\includegraphics[width=1\linewidth]{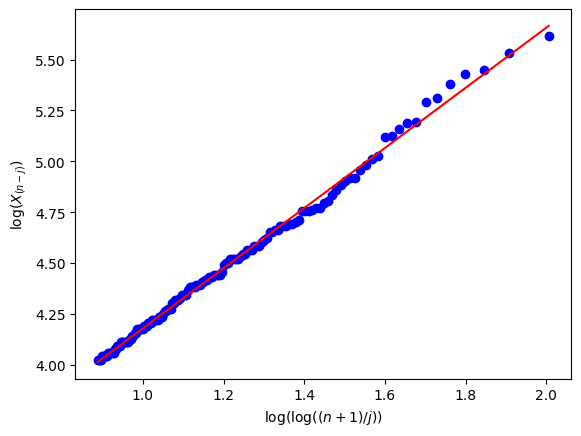}}
\end{minipage}
\caption{{\footnotesize Hill (left), log-Weibull (center) and Weibull (right) plots of the largest $k=150$ daily precipitation observations in Jena, Germany, in April during 1901--2021, $n=1693,$ the red solid lines correspond to the best linear fits to these data.}}
\label{saentis}
\end{figure}

\section{Proofs}\label{proofs}

\subsection{Proof of Proposition \ref{P1}
}

 The continuity of $G$ and infinity of its right endpoint imply that for all $c>1$ there exists some $\tilde c>1, \tilde c = \tilde c(t),$ such that $c = u_G(\tilde c t)/u_G(t).$ From $C_\delta(H,G)$ it follows that
\[u_H(t) \ge u_G(t) \frac{u_H\big(\tilde{c}^{1+\delta}t\big)}{u_G(\tilde ct)}.\] Using this relation, we derive
\begin{eqnarray*}\frac{\big(1 - H(c u_H(t))\big)^{1-\varepsilon}}{1 - G(c u_G(t))} & = & \frac{\left(1 - H\Big(\frac{u_G(\tilde c t)}{u_G(t)} u_H(t)\Big)\right)^{1-\varepsilon}}{1 - G(u_G(\tilde c t))}\\ &\le&
\frac{\left(1 - H\Big(\frac{u_G(\tilde c t)}{u_G(t)} \frac{u_H(\tilde{c}^{1+\delta}t)}{u_G(\tilde ct)} u_G(t)\Big)\right)^{1-\varepsilon}}{1 - G(u_G(\tilde c t))} \\
& = & \frac{\left(1 - H(u_H(\tilde{c}^{1+\delta}t))\right)^{1-\varepsilon}}{1 - G(u_G(\tilde c t))} = \frac{(\tilde{c}^{1+\delta}t)^{-(1-\varepsilon)}}{(\tilde c t)^{-1}} = t^\varepsilon \\ & = & \frac{(1 - H(u_H(t)))^{1-\varepsilon}}{1 - G(u_G(t))},\end{eqnarray*}
whence the proposition follows.

\subsection{Proof of Example \ref{E1}}
First let us show that for all $H\in A_1$ the condition $C_\delta(F_0, H)$ holds for some $\delta>0$. We have $u_{F_0}(t) = \exp\big((\ln\ln t/b)^2\big).$ By its definition \eqref{deflogw}, the function $\ln (1 - H(e^t))$ is regularly varying at infinity with index $1/\theta,$ $\theta>0,$ thus $1 - H(t) = \exp\big(-(\ln t)^{1/\theta}\ell(\ln t)\big), t>1,$ for some slowly varying at infinity function $\ell(t).$ Therefore, by properties of slowly varying functions (see, e.g., \cite{bingham}), there is a slowly varying function $\ell^{\sharp}(t)$ such that $u_{H}(t) = \exp\big((\ln t)^\theta\ell^\sharp(\ln t)\big).$

So, to check the condition $C_\delta(F_0, H),$ we show that for some $\delta>0$ there is a $t_0$ such that for all $t>t_0$ the ratio \[u_{F_0}(c^{1+\delta}t)/u_H(ct) = \exp\big((\ln\ln(c^{1+\delta}t)/b)^2 - (\ln(ct))^\theta \ell^\sharp(\ln (ct)) \big) = \exp(f(c, t))\] does not increase with respect to $c>1.$ For this purpose it is enough to show that the derivative with respect to $c$ of the function in the exponent on the right-hand side of the latter relation is negative for all $t>t_0$ and $c>1.$ Below we assume that the derivative of $\ell^\sharp$ is eventually monotone. If the slowly varying function $\ell(t)$ is differentiable and its derivative is eventually monotone, then \begin{equation}\lim_{t\to\infty} t \ell^\prime(t)/\ell(t) = 0,\label{slowv}\end{equation} see, e.g., Theorem 1.7.2, \cite{bingham}. Thus, for large $t_0$ we have
\begin{eqnarray}\nonumber&&\frac{\partial f(c,t)}{\partial c} = \frac{1}{c}\, \frac{2(1+\delta)}{b^2} \,\frac{\ln\ln(c^{1+\delta}t)}{\ln(c^{1+\delta}t)} - \frac{1}{c}\theta (\ln(ct))^{\theta-1} \ell^\sharp(\ln (ct))(1 + o(1))\\ \nonumber
&&< t \left(\frac{1}{ct} \, \frac{2(1+\delta)}{b^2}\,\frac{\ln\ln(ct)}{\ln(ct)} - \frac{1}{ct}\theta (\ln(ct))^{\theta-1} \ell^\sharp(\ln (ct))(1 + o(1))\right)\\ &&= t\,\tilde f^\prime(ct),\label{deriv}\end{eqnarray}
with \[\tilde f(x) = (1+\delta)((\ln\ln x)/b)^2 - (\ln x)^\theta \ell^\sharp(\ln x).\] It is easy to see that there exists $x_0$ such that for all $x>x_0$ the derivative of $\tilde f(x)$ is negative. Selecting $t_0 = x_0,$ we finally derive by \eqref{deriv} that $\partial f(c,t)/\partial c<0$ for all $t>t_0$ and $c>1.$

Checking the condition $C_0(G, F_0)$ for $G\in A_0,$ where $A_0$ is the class of tails of Weibull-type distributions, is much easier. Indeed, using properties of slowly varying functions and the definition of Weibull-type distributions (see formula (4)), we obtain $u_G(t) = (\ln t)^\theta \ell^\sharp(\ln t)$ for $\theta>0$ and some slowly varying $\ell^\sharp(t).$ Then we need to show that the function
\[u_G(t)/u_{F_0}(t) = \exp\big(\theta\ln\ln t + \ln\ell^\sharp(\ln t) - ((\ln\ln t)/b)^2\big)\] does not increase eventually. Assuming again that the derivative of $\ell^\sharp$ is eventually monotone, we can show this by proving that the derivative of the function in the exponent on the right-hand side of the latter relation is eventually negative. This follows immediately by using \eqref{slowv}.

\subsection{Proof of Theorem \ref{T1}
}
We have
\[\sqrt{k}(\tilde R_{k,n} - 1) = \sqrt{k}(\tilde R_{k,n} - R_{k,n}) + \sqrt{k}(R_{k,n} - 1),\]
where $R_{k,n}$ is defined by \eqref{initialrkn}.
First, $-\ln (1 - F_0(X_{(n-i)})) \stackrel{d}{=} E_{(n-i)},$ $i=0, \ldots, n-1,$ where $E_{(1)}\le \cdots\le E_{(n)}$ denote the order statistics of independent standard exponential random variables $E_1, \ldots, E_n.$ Then, by the R\'enyi representation \cite{renyi}
\begin{equation}R_{k,n} \stackrel{d}{=} \frac{1}{k}\sum_{i=1}^{k} \big(E_{(n-k+i)} - E_{(n-k)} \big)\stackrel{d}{=} \frac{1}{k}\sum_{i=1}^{k} E^\ast_{(i)} = \frac{1}{k}\sum_{i=1}^{k} E^\ast_{i},\label{rknexp}\end{equation} where $\{E^\ast_{(i)}\}_{i=1}^k$ are the order statistics of independent standard exponential random variables $\{E^\ast_i\}_{i=1}^k.$ Thus by the central limit theorem,
\begin{equation}\label{old}
\sqrt{k}(R_{k,n} - 1) \xrightarrow{d} N(0,1),
\end{equation} see also the proof of Theorem 1 in \cite{R1}.

So to complete the proof, it is enough to show that
\[\sqrt{k}(\tilde R_{k,n} - R_{k,n}) \xrightarrow{P} 0.\]
To prove this, we use the Chebyshev inequality for $\sqrt{k}(\tilde R_{k,n} - R_{k,n})$ given $X_{(n-k)} = q$ and the law of total probability. For this purpose, we evaluate the asymptotic behaviour of $\mathbb{E}(\tilde R_{k,n} - R_{k,n} \mid X_{(n-k)})$ and ${\rm Var}(\tilde R_{k,n} - R_{k,n} \mid X_{(n-k)}).$ \\
\\
{\bf Explicit form of $\mathbb{E}(\tilde R_{k,n} - R_{k,n}) \mid X_{(n-k)} = q).$}\\
\\
We have
\begin{eqnarray}
\nonumber\tilde{R}_{k,n} - R_{k,n} &=&
 \left(\ln\frac{k}{n} - \ln(1 - F_0(X_{(n-k)})) \right)\\ \nonumber && - \frac{1}{k}\sum_{i=n-k+1}^{n}\ln\frac{1-F_0(u_0(n/k) X_{(i)}/X_{(n-k)})}{1
- F_0(X_{(i)})}.
\end{eqnarray}
Consider the conditional distribution of $\tilde{R}_{k,n} - R_{k,n}$ given $X_{(n-k)}=q.$ By Lemma 3.4.1 in \cite{dehaan}, the joint conditional distribution
of the set of order statistics $\{X_{(i)}\}_{i=n-k+1}^{n}$ given $X_{(n-k)}=q$ coincides with the (unconditional) joint distribution of the set of order statistics $\{X_{(i)}^{\ast}\}_{i=1}^{k}$ of i.i.d. random variables $\{X_{i}^{\ast}\}_{i=1}^{k}$ with common cdf
\[ F_q(x)=P(X\leq x \mid X>q)= \frac{F_{0}(x)-F_0(q)}{1-F_{0}(q)},\quad x>q.\]
Thus, given $X_{(n-k)}=q$,
\begin{eqnarray}\tilde{R}_{k,n} - R_{k,n} \stackrel{d}{=}
\left(\ln\frac{k}{n} -
\ln(1 - F_0(q))\right) -
\frac{1}{k}\sum_{i=1}^{k}\ln\frac{1-F_0(u_0(n/k) X^\ast_{i}/q)}{1
- F_0(X^\ast_{i})}. \;\; \label{R}\end{eqnarray} Hereafter we write $u_0$ instead of $u_0(n/k)$ for convenience. For $Y^\ast :=
\ln(1-F_0(u_0 X^\ast_{1}/q)) - \ln(1 - F_0(X^\ast_{1})),$ after direct, but tedious calculations, we derive
\begin{equation}\mathbb{E}Y^\ast = \ln\frac{1-F_0(u_0)}{1-F_0(q)} - (f(q) - 1),\label{E}\end{equation}
where \begin{equation}f(q) = f(q;n,k) =
\int_{q}^{\infty}\frac{u_0}{q}\frac{1-F_0(x)}{1-F_0(q)}\frac{p_0(u_0x/q)}{1-F_0(u_0x/q)}dx,\label{f}\end{equation}
and $p_0(t)$ denotes the pdf of $F_0.$
This and \eqref{R} imply that \[\mathbb{E}(\tilde{R}_{k,n} -
R_{k,n}\mid X_{(n-k)} = q) = f(q) - f(u_0) = f(q) - 1.\]
\\
{\bf Asymptotics of $\sqrt{k}\mathbb{E}(\tilde R_{k,n} - R_{k,n}\mid X_{(n-k)}).$}\\
\\
Next, we show that
\begin{equation}\sqrt{k}\mathbb{E}(\tilde R_{k,n} - R_{k,n}\mid X_{(n-k)}) = \sqrt{k}(f(X_{(n-k)}) - f(u_0))\xrightarrow{P}
0.\label{consist}\end{equation} As long as $F_0(u_0(x)) = 1 - 1/x,$
we have $p_0(u_0(x))u_0^\prime(x) = x^{-2}$ and $(xu^\prime_0(x))^{-1} = x
p_0(u_0(x)).$ Lemma 2.2.1 in \cite{dehaan} implies that
\begin{equation}\frac{n}{\sqrt{k}}p_0(u_0)(X_{(n-k)} -
u_0)\stackrel{d}{\rightarrow} N(0,1).\label{delta}\end{equation}
Applying the delta method to \eqref{delta}, we obtain
\begin{equation}\frac{n}{\sqrt{k}}p_0(u_0)\frac{f(X_{(n-k)}) -
f(u_0)}{f^{\prime}(u_0)}\xrightarrow{d} N(0,1).\label{delta1}\end{equation}
This is not the classical delta method, since $u_0$ is not a constant and so \eqref{delta1} should be established. By the mean value theorem, we have $f(X_{(n-k)}) - f(u_0) = f^\prime(c)(X_{(n-k)} - u_0)$ a.s., where $c$ is some (random) value between $X_{(n-k)}$ and $u_0.$ Thus it is enough to show that $f^\prime(c)/f^\prime(u_0) \xrightarrow{P} 1.$ The proof uses the explicit form of $f^\prime$ provided below and is based on the multiple application of Condition \ref{Cond1}
 and \eqref{delta}; it is very technical and is not of special interest. Thus here and in similar situations below we omit such technical details.

Next, calculating directly $f^\prime(q)$ and substituting $q = u_0,$ we derive
\begin{equation*}
f^\prime(u_0)  =  \frac{n}{k}\bigg( p_0(u_0) - \int_{u_0}^\infty\frac{x}{u_0}\frac{p^2_0(x)}{1 -
F_0(x)}dx \bigg).
\end{equation*}
For absolutely continuous distributions with infinite right endpoint belonging to $\mathcal{D}(EV_\gamma),$
\begin{equation}\frac{1 - F_0(u)}{up_0(u)} \to \gamma, \quad u\to\infty,\label{easy}\end{equation}
see, e.g., Theorem 1.1.11 in \cite{dehaan}, for $\gamma>0$ and p. 18 ibid. for $\gamma=0.$
Using the latter, L'H$\hat{\text{o}}$pital's rule and Condition \ref{Cond1},
we get
\begin{eqnarray}\nonumber \lim_{u\to\infty} \frac{u p_0(u)}{\int_u^\infty x \frac{p^2_0(x)}{1 -
F_0(x)}dx}  &=&  -\lim_{u\rightarrow \infty} \frac{up^\prime_0(u) + p_0(u)}{u p^2_0(u)/(1 -
F_0(u))}\\  &=& \label{lopital} -\lim_{u\rightarrow \infty} \frac{p^\prime_0(u)(1-F_0(u))}{(p_0(u))^2} - \lim_{u\rightarrow \infty} \frac{1-F_0(u)}{up_0(u)}\\ &=& 1,\nonumber\end{eqnarray}
whence \[\frac{k}{n p_0(u_0)} f^\prime(u_0) \xrightarrow{P} 0.\] Combining the latter and \eqref{delta1}, we derive \eqref{consist}.\\
\\
{\bf Asymptotics of ${\rm Var}(\tilde R_{k,n} - R_{k,n}\mid X_{(n-k)}).$}
\\
\\
First, we have,
\begin{eqnarray*}{\rm Var}(\tilde{R}_{k,n} - R_{k,n} \mid X_{(n-k)} = q) &=&
\frac{1}{k}{\rm Var}\left(\ln \frac{k}{n} - \ln(1-F_0(q)) -
Y^\ast\right)\\ &=& \frac{1}{k}{\rm Var} (Y^\ast).\end{eqnarray*}
Integrating by parts, we derive for the second moment of $Y^\ast$
\begin{eqnarray*}\mathbb{E}(Y^\ast)^2 = \int_{q}^{\infty}
\left(\ln\frac{1-F_0(u_0x/q)}{1-F_0(x)}\right)^2
d\left(\frac{F_0(x)-F_0(q)}{1-F_0(q)}\right) = \left(\ln \frac{1 -
F_0(q)}{1 - F_0(u_0)}\right)^2 \\ + 2\int_q^\infty
\ln\frac{1-F_0(u_0x/q)}{1-F_0(x)} \left(\frac{p_0(x)}{1 - F_0(x)} -
\frac{u_0}{q}\frac{p_0(u_0x/q)}{1-F_0(u_0x/q)}\right) \frac{1 -
F_0(x)}{1-F_0(q)}dx \\
=:  \left(\ln \frac{1 -
F_0(q)}{1 - F_0(u_0)}\right)^2 + 2 h(q).
\end{eqnarray*}
From the latter and \eqref{E},
\[\mathbb{E}(Y^\ast)^2 - (\mathbb{E}Y^\ast)^2 = 2h(q) - 2 (f(q)-1)\ln \frac{1 -
F_0(q)}{1 - F_0(u_0)} - (f(q)-1)^2 =: V(q).\]
Fix $\delta:$
$0<\delta<1/2.$ Let us show that \begin{equation}\label{asvar} k^{2\delta}V(X_{(n-k)}) \xrightarrow{P} 0.\end{equation} Note that $V(X_{(n-k)}) = k {\rm Var}(\tilde R_{k,n} - R_{k,n}\mid X_{(n-k)}).$
Let $E_1, \ldots, E_n$ be i.i.d. standard exponential variables and $E_{(1)} \le \cdots \le E_{(n)}$ be their order statistics. By continuity of $F_0,$ we get $-\ln(1 - F_0(X_{(n-k)})) \stackrel{d}{=} E_{(n-k)},$ thus
\begin{eqnarray*}
 -\sqrt{k}\ln \left(\frac{1 - F_0(X_{(n-k)})}{1-F_0(u_0)}\right) \stackrel{d}{=} \sqrt{k} \big(E_{(n-k)} - \ln(n/k)\big) \xrightarrow{d} N(0,1),\end{eqnarray*} by Lemma 2.2.1 in \cite{dehaan}.
From the latter, \eqref{consist} and the relation $f(u_0)=1,$ it follows that
\[k^{2\delta} \left(2(f(X_{(n-k)})-1)\ln \frac{1 - F_0(X_{(n-k)})}{1 - F_0(u_0)} +
(f(X_{(n-k)})-1)^2\right) \xrightarrow{P} 0.\] To prove \eqref{asvar} it remains to show that $k^{2\delta}h(X_{(n-k)}) \xrightarrow{P} 0.$ For this purpose we apply the delta method again. In this situation we cannot use the version of the delta method which we applied above since $h^\prime(u_0) = 0,$ so we aim at finding the second derivative of
$h(q)$ and substituting $q=u_0.$ After tedious calculations we obtain \[h^{\prime\prime}(u_0) =
\frac{1}{u_0^2(1 - F_0(u_0))}\int_{u_0}^\infty \frac{x^2p^3_0(x)}{(1 - F_0(x))^2}dx.\]
Note also that $h(u_0) = 0.$ From \eqref{delta} and the relation $1 - F_0(u_0) = k/n$
we obtain \begin{eqnarray}\frac{n^2}{k}
\frac{p_0^2(u_0)(h(X_{(n-k)})-h(u_0))}{2h^{\prime\prime}(u_0)} =
\frac{n p_0^2(u_0) h(X_{(n-k)})}{\int_{u_0}^\infty\frac{2x^2}{u_0^2}\frac{p_0^3(x)}{(1-F_0(x))^2}dx} \xrightarrow{d}
\xi^2,\label{almost_last}\end{eqnarray} where $\xi\sim N(0,1)$ and the asymptotics of the integral in the denominator follows from the L'H$\hat{\text{o}}$pital rule, Condition \ref{Cond1}
 and \eqref{easy},
\begin{eqnarray*}\lim_{u\to\infty}\frac{u^2 p_0^2(u)/(1-F_0(u))}{\int_{u}^\infty 2x^2\frac{p_0^3(x)}{(1-F_0(x))^2}dx} = \frac{1}{2}.\end{eqnarray*}
Substituting the latter into \eqref{almost_last} and using the Slutsky theorem, we derive
\[ k h(X_{(n-k)}) \xrightarrow{d}
\xi^2.\] Therefore, for $\delta\in(0, 1/2)$ it holds $k^{2\delta} h(X_{(n-k)}) \xrightarrow{P} 0,$
that proves \eqref{asvar}.
\\
\\
{\bf Application of the Chebyshev inequality.}
\\
\\
To complete the proof, we must show that
\begin{equation}T_{k,n} :=\;\mid\sqrt{k}(\tilde{R}_{k,n} - R_{k,n}) - \sqrt{k}(f(X_{(n-k)}) -
f(u_0))\mid\; \xrightarrow{P} 0,\label{P}\end{equation} that, together with \eqref{consist}, implies the theorem. By the Chebyshev inequality, for every $\varepsilon>0$ we get
\begin{eqnarray*}P\big( T_{k,n} >\varepsilon \mid X_{(n-k)} = q\big)
\leq \frac{k {\rm Var}(\tilde{R}_{k,n} -
R_{k,n}\mid X_{(n-k)} = q)}{\varepsilon^2} = \frac{V(q)}{\varepsilon^2}.\end{eqnarray*} By the latter and the law of total probability, we finally derive
\begin{eqnarray*}
P\left(T_{k,n} > \varepsilon\right) &=& P\left(T_{k,n} > \varepsilon,\; k^\delta\sqrt{V(X_{(n-k)})} \le  \varepsilon \right)\\ &&+\; P\left(T_{k,n} > \varepsilon,\; k^\delta \sqrt{V(X_{(n-k)})} > \varepsilon \right)\\
&\le& P\left(T_{k,n} > k^\delta \sqrt{V(X_{(n-k)})}\right) + P\left(k^\delta \sqrt{V(X_{(n-k)})} > \varepsilon \right)\\
&=& \int_{\mathbb{R}} P\left(T_{k,n} > k^\delta \sqrt{V(X_{(n-k)})} \mid X_{(n-k)} = q\right) p_{X_{(n-k)}}(q) dq\\&& + P\left(k^\delta \sqrt{V(X_{(n-k)})} > \varepsilon \right) \\
&\le & \int_{\mathbb{R}} k^{-2\delta} p_{X_{(n-k)}}(q) dq + P\left(k^\delta \sqrt{V(X_{(n-k)})} > \varepsilon \right)
\\ &=& k^{-2\delta} + P\left(k^\delta \sqrt{V(X_{(n-k)})} > \varepsilon \right),
\end{eqnarray*}
where $p_{X_{(n-k)}}(q)$ is the pdf of
$X_{(n-k)}.$ The result follows from \eqref{asvar}.

\subsection{Proof of Theorem \ref{T2}
}

To prove Theorem \ref{T2},
we need the following result.

\begin{lemma}\label{L1} Let $X_1^0, \ldots, X_n^0$ be i.i.d. random variables with common cdf $F_0.$ Then under the conditions of Theorem \ref{T2}
\begin{equation}\label{appr}
\frac{1-F_{0}(c X^0_{(n-k)})}{1-F_{0}(X^0_{(n-k)})} - \frac{1-F_{0}(c u_0)}{1-F_{0}(u_0)} = O_P(1/\sqrt{k})
\end{equation} uniformly in $c>1.$
\end{lemma}
\begin{proof}
Applying the delta method for the function $f(x) = (1 - F_0(cx))/(1 - F_0(x))$ to the relation \eqref{delta}, we derive
\[\frac{n}{\sqrt{k}} p_0(u_0)\frac{f(X^0_{(n-k)}) - f(u_0)}{f^\prime(u_0)} \xrightarrow{d} N(0,1).\] Thus to prove \eqref{appr} it is enough to show that
\begin{equation}\frac{\sqrt{k}}{n}\frac{f^\prime(u_0)}{p_0(u_0)} = O(1/\sqrt{k})\label{appr2}\end{equation} uniformly in $c>1.$
As long as
\[f^\prime(u_0) = \left(\frac{1 - F_0(cx)}{1 - F_0(x)}\right)^\prime_{x=u_0} = \frac{p_0(u_0)(1 - F_0(c u_0)) - cp_0(cu_0) (1 - F_0(u_0))}{(1 - F_0(u_0))^2},\]
then using $1 - F_0(u_0) = k/n$ we can simplify the relation \eqref{appr2} to
\begin{equation}\label{fdens} \frac{1 - F_0(c u_0)}{1 - F_0(u_0)} - c\frac{p_0(c u_0)}{p_0(u_0)} = O(1)
\end{equation}
uniformly in $c>1.$ Clear, $\overline{F}_0(cu_0)/\overline{F}_0(u_0) \le 1$ for all $c>1.$ If $F_0\in \mathcal{D}(G_\gamma)$ with $\gamma>0,$ boundness from above of the ratio $cp_0(c u_0)/p_0(u_0)$ follows from Potter's inequality, see, e.g., Proposition B.1.9 in \cite{dehaan}. If $\gamma=0,$ then the required fact follows from formula (1.1.33) ibid., that holds under Condition \ref{Cond1},
 relation $p_0(u_0(t)) u_0^\prime(t) = t^{-2}$ and Potter's inequality again. Hence \eqref{fdens} and \eqref{appr} hold. \end{proof}

Let us show that under the condition $C_\delta(F_0, F_1),$ $\delta>0,$ \[\sqrt{k}(\tilde{R}_{k,n}-1)\gg\sqrt{k}\left(  \frac{1}{k}\sum
_{i=1}^{k}E_{i}-1\right)\] for some i.i.d. random variables $\{E_i\}_{i\ge 1}$ with mean more than 1, and the result follows from the central limit theorem. Here $X\gg Y$ means that $X$ is stochastically larger than $Y.$ The proof remains almost the same if $C_\delta(F_1, F_0)$ holds with $\delta>0$ instead of $C_\delta(F_0, F_1)$.\\
\\
 Let $X_1, \dots, X_n$ be i.i.d. random variables with cdf $F_1.$ Consider the asymptotic behavior of the statistic
$\tilde{R}_{k,n}$ as $n\rightarrow\infty.$ Denote $Y_{i}= \ln
(k/n) - \ln(1-F_0(u_0X^\ast_{i}/q)),$ $i=1,\ldots, k,$ where
$\{X_{i}^{\ast}\}_{i=1}^{k}$ are i.i.d. random variables defined in the same way as in Theorem 2 
 with common cdf
\[
F_{q}^1(x)=\frac{F_{1}(x)-F_{1}(q)}{1-F_{1}(q)},\ \ q<x.
\]
By Lemma 3.4.1 in \cite{dehaan}, the joint distribution of the order statistics $\{Y_{(i)}\}_{i=1}^{k}$
of $\{Y_{j}\}_{i=1}^{k}$ coincides with the conditional joint distribution of the set of random variables
 $\{Z_{j}\}_{j=1}^{k}$ given $X_{(n-k)}=q,$ with
\[
Z_{j}= \ln (k/n) - \ln(1-F_{0}(u_0 X_{(n-j+1)}/X_{(n-k)})),\
j=1,...,k.
\]
Clearly, $\tilde{R}_{k,n}=\frac{1}{k}\sum_{j=1}^{k}Z_{j}.$
Thus, the conditional distribution of $\tilde{R}_{k,n}$ given
$X_{(n-k)}=q$ coincides with the distribution of
$\frac{1}{k}\sum_{i=1}^{k}Y_{i}.$
Therefore,
\begin{eqnarray} \nonumber P(Y_1\leq x) &=& P(\ln(1 - F_0(u_0 X^\ast_1/q))\geq \ln(k/n) - x) \\
&=& P\big(X_1^\ast\leq q u_0^{-1} F_0^\leftarrow(1 - ke^{-x}/n)\big) \label{T2main} \\ &=&
\frac{F_1\big(q u_0^{-1} F_0^\leftarrow(1 - ke^{-x}/n)\big) - F_1(q)}{1 -
F_1(q)}. \nonumber 
\end{eqnarray}

By assumption, the cdfs
$F_{1}$ and $F_{0}$ satisfy either the condition $C_\delta(F_{0},F_{1})$ or
$C_\delta(F_{1},F_{0})$ with $\delta>0$ by assumptions. Assume that $C_\delta(F_{0},F_{1})$ holds with some
 $\delta>0$ and $t_{0},$ if $C_\delta(F_{1},F_{0})$ is satisfied the proof is similar.
If $x^{\ast}=+\infty,$ then $X_{(n-k)}\rightarrow+\infty$ almost surely, so we can deal only with the case $n/k>t_{0}.$ Denote the quantile function of $F_1$ by $u_1(t)$ and write $u_1$ instead of $u_1(n/k).$ By Proposition 1,
\begin{equation}
\frac{1-F_{1}(c u_1)}{1-F_{1}(u_1)}\geq\frac{(1-F_{0}(c u_0))^{1-\varepsilon}%
}{(1-F_{0}(u_0))^{1-\varepsilon}},\ \ c>1,
\label{ineq}\end{equation} for some $\varepsilon>0.$

Denote $r = u_1^\leftarrow(q).$ Note also that if $X_{(n-k)} = u_1(\xi)$ for some random $\xi,$ then  $u_0(\xi) \stackrel{d}{=} X_{(n-k)}^0$ with $\{X_i^0\}_{i=1}^n$ introduced in Lemma \ref{L1}. Let us rewrite \eqref{T2main}, using \eqref{ineq} with $c = u_0^{-1} F_0^\leftarrow(1- ke^{-x}/n)$ and \eqref{appr}. Uniformly in $x>0,$ we have
\begin{eqnarray*}P(Y_1\leq x) & = & 1 - \frac{1 - F_1\big(u_1(r) u_0^{-1} F_0^\leftarrow(1
- ke^{-x}/n)\big)}{1 - F_1(u_1(r))} \\
& \leq &
1 - \frac{\big(1 - F_0(u_0(r) u_0^{-1} F_0^\leftarrow(1 -
ke^{-x}/n))\big)^{1-\varepsilon}}{\big(1 -
F_0(u_0(r))\big)^{1-\varepsilon}}\\
& = & 1 - \left(\frac{1 - F_0(F_0^\leftarrow(1 -
ke^{-x}/n))}{1 -
F_0(u_0)} + O(1/\sqrt{k})\right)^{1-\varepsilon}\\
& = & 1 - \left(e^{-x} + O(1/\sqrt{k})\right)^{1-\varepsilon} =  1 - e^{-(1-\varepsilon) x} + O(1/\sqrt{k}).
\end{eqnarray*}
Since $O(1/\sqrt{k})$ vanishes, for $k$ large enough there exists a cdf $G(x)$ such that
\[ G(x) \ge \min(1 - e^{-(1-\varepsilon) x} + O(1/\sqrt{k}), 1) \] with mean $(1 - \varepsilon_1)^{-1},$  $\varepsilon_1\in (0, \varepsilon),$ and variance $\sigma^2>0.$ Hence, $Y_{1}$ is stochastically larger than a random variable $E$ with cdf
$G.$ Next, let
$E_{1},\ldots,E_{k}$ be i.i.d. random variables with common cdf $G,$ then
\begin{equation}
\sqrt{k}\left(  \frac{1}{k}\sum_{i=1}^{k}Y_{i}-1\right)  \gg
\sqrt{k}\left(  \frac{1}{k}\sum_{i=1}^{k}E_{i}-1\right)
.\label{ll}%
\end{equation}
Provided \eqref{ll} holds for all $q>x_{0},$ 
\begin{equation}
\sqrt{k}(\tilde{R}_{k,n}-1)\gg\sqrt{k}\left(  \frac{1}{k}\sum
_{i=1}^{k}E_{i}-1\right)  .\label{LL}%
\end{equation}
By the central limit theorem,
\[
\sqrt{k}\left(  \frac{1}{k}\sum_{i=1}^{k}%
E_{i}-\frac{1}{1-\varepsilon_1}\right) \xrightarrow{d} N(0,\sigma^2),
\]
so \[\sqrt{k}\left(  \frac{1}{k}\sum_{i=1}^{k}E_{i}-1\right)
\xrightarrow{P}+\infty,\] which with \eqref{LL} completes the proof of
Theorem \ref{T2}.

\subsection{Proof of Corollary \ref{C1}
}

Let us assume that $C_0(F_0, F_1)$ holds; the proof under $C_0(F_1, F_0)$ is similar. Condition $C_0(F_0, F_1)$ implies that
\begin{equation}\label{fromc0}\frac{u_0(ct)}{u_0(t)} \le \frac{u_1(ct)}{u_1(t)}\end{equation} for all $c\ge 1$ and $t\ge t_0.$ As mentioned in the proof of Theorem \ref{T2},
 under the assumptions of Corollary \ref{C1},
 $X^0_1 \stackrel{d}{=} u_0(u_1^\leftarrow(X_{1})) =: V_{1}.$ Hence substituting $t = u_1^{\leftarrow}(X_{(n-k)})$ and $c = u_1^{\leftarrow}(X_{(i)})/u_1^{\leftarrow}(X_{(n-k)})$ in \eqref{fromc0}, we get for all $i\in \{n-k+1, \ldots, n\}$ and $n$ large enough
\begin{equation} \label{v}\frac{V_{(i)}}{V_{(n-k)}} = \frac{u_0(u_1^{\leftarrow}(X_{(i)}))}{u_0(u_1^{\leftarrow}(X_{(n-k)}))}
\le \frac{u_1(u_1^{\leftarrow}(X_{(i)}))}{u_1(u_1^{\leftarrow}(X_{(n-k)}))} = \frac{X_{(i)}}{X_{(n-k)}},
\end{equation} where $V_{(1)} \le \cdots \le V_{(n)}$ are the order statistics of i.i.d. random variables $\{V_i\}_{i=1}^n,$ $V_i = u_0(u_1^\leftarrow(X_{i}))$ for $i\in \{1, \ldots, n\}.$ Since $(V_{(1)}, \ldots, V_{(n)}) \stackrel{d}{=}\\ (X_{(1)}^0, \ldots, X_{(n)}^0),$ it follows from \eqref{v} that
\begin{eqnarray*}\tilde R^0_{k,n} &\stackrel{d}{=}& \ln\frac{k}{n} -
\frac{1}{k}\sum_{i=n-k+1}^{n}\ln(1-F_0(u_0V_{(i)}/V_{(n-k)}))\\ &\le& \ln\frac{k}{n} -
\frac{1}{k}\sum_{i=n-k+1}^{n}\ln(1-F_0(u_0X_{(i)}/X_{(n-k)})) = \tilde R_{k,n}.\end{eqnarray*} The result follows.

\subsection{Proof of Theorem \ref{T3}
}

To prove Theorem \ref{T3},
 we need the following result.

\begin{lemma}\label{L2} Let $E_1, \ldots, E_n$ be i.i.d. standard exponential with $E_{(1)}\le \cdots\le E_{(n)},$ their order statistics. Then
\begin{equation}\label{2exp}
\frac{n}{\sqrt{k}}\left(\left(\begin{array}{c}E_{(k)} \\ E_{(2k)}\end{array}\right) +
\left(\begin{array}{c} \ln(1 - k/n) \\ \ln(1 - 2k/n) \end{array}\right)\right) \xrightarrow{d} N\left(\Big(\begin{array}{c}0 \\ 0\end{array}\Big), \Big(\begin{array}{cc} 1 & 1 \\ 1 & 2 \end{array}\Big)\right),
\end{equation} if $k\to\infty,$ $k/n\to 0$ as $n\to\infty.$
\end{lemma}
\begin{proof}
The result of the lemma is almost the direct consequence of the following two-dimensional extension of Smirnov's lemma \cite{smirnov}, see also Lemma 2.2.3 in \cite{dehaan}. Let $U_{(1)} \le \cdots \le U_{(n)}$ be order
statistics from a standard uniform sample. Then, if $k\to\infty,$ $k/n\to 0$ as $n\to\infty,$
\begin{equation}\label{2smirnov}\left(\begin{array}{c} \frac{U_{(n-k)} - b_{n,k}}{a_{n,k}}\\ \frac{U_{(n-2k)} - b_{n,2k}}{a_{n,2k}}\end{array} \right)  \xrightarrow{d}  N\left(\Big(\begin{array}{c}0 \\ 0\end{array}\Big), \Big(\begin{array}{cc} 1 & 1/\sqrt{2} \\ 1/\sqrt{2} & 1 \end{array}\Big)\right),\end{equation}
where for $t\in\mathbb{N},$ $t<n,$ $b_{n,t} = (n-t)/(n-1)$ and $ a_{n,t} = \sqrt{b_{n,t}(1 - b_{n,t})/(n-1)}.$
The proof of \eqref{2smirnov} is just a repetition of that of Lemma 2.2.3 in \cite{dehaan}. Firstly, the pdf of the random vector $(U_{(n-k)}, U_{(n-2k)})$ is
\[\frac{n!}{((k-1)!)^2 (n-2k)!}(1 - x)^{k-1} (x-y)^{k-1} y^{n-2k},\] thus the pdf of the vector $\big((U_{(n-k)} - b_{n,k})/a_{n,k},\ (U_{(n-2k)} - b_{n,2k})/a_{n,2k} \big)$ is
\begin{eqnarray*}\frac{n!}{((k-1)!)^2 (n-2k)!}\,a_{n,k}\,a_{n,2k} (1 - b_{n,k})^{k-1} (b_{n,k} - b_{n,2k})^{k-1} (b_{n,2k})^{n-2k} \\ \times\Big(1 - \frac{a_{n,k} x}{1 - b_{n,k}}\Big)^{k-1} \Big(1 + \frac{a_{n,k} x - a_{n,2k} y}{b_{n,k} - b_{n,2k}}\Big)^{k-1}\Big(1 + \frac{a_{n,2k} x}{b_{n,2k}}\Big)^{n-2k}.\end{eqnarray*}
By the Stirling formula one can conclude that the expression on the first line of the latter formula tends to $ (\sqrt{2}\pi)^{-1},$ whereas the expression on the second line tends to
\[\exp( - x^2 +\sqrt{2}xy - y^2).\] Thus, \eqref{2smirnov} follows from the fact that pointwise convergence of the sequence of pdfs implies weak convergence of the probability distributions (Scheffe's theorem).

Next, noticing that $a_{n,k} \sim a_{n,2k}/\sqrt{2} \sim \sqrt{k}/n$ and using the properties of multivariate normal distribution, we have
\begin{equation}\frac{n}{\sqrt{k}}\left(\left(\begin{array}{c}U_{(n-k)} \\ U_{(n-2k)}\end{array}\right) -
\left(\begin{array}{c}\frac{n-k}{n-1} \\ \frac{n-2k}{n-1} \end{array}\right)\right) \xrightarrow{d} N\left(\Big(\begin{array}{c}0 \\ 0\end{array}\Big), \Big(\begin{array}{cc} 1 & 1 \\ 1 & 2 \end{array}\Big)\right).\label{2dim}\end{equation}
Notice that $(E_{(k)}, E_{(2k)})\stackrel{d}{=} (- \ln(U_{(n-k)}), - \ln(U_{(n-2k)}))$ and denote $f(x) =  - \ln x.$ We derive by the mean value theorem
\begin{eqnarray*}  \frac{n}{\sqrt{k}}\left(\left(\begin{array}{c}E_{(k)} \\ E_{(2k)}\end{array}\right) +
\left(\begin{array}{c} \ln\big(\frac{n-k}{n-1}\big) \\ \ln\big(\frac{n-2k}{n-1}\big) \end{array}\right)\right) &\stackrel{d}{=}& \frac{n}{\sqrt{k}}\left(\begin{array}{c}f(U_{(n-k)}) - f\big(\frac{n-k}{n-1}\big) \\ f(U_{(n-2k)}) - f\big(\frac{n-2k}{n-1}\big)\end{array}\right) \\ &=& \frac{n}{\sqrt{k}}\left(\begin{array}{c}f^\prime(c_1)\big(U_{(n-k)} - \frac{n-k}{n-1}\big) \\ f^\prime(c_2)\big(U_{(n-2k)} - \frac{n-2k}{n-1}\big) \end{array}\right)\end{eqnarray*}  where $c_i,$ $i=1,2,$ are
some (random) values between $U_{(n-ik)}$ and $(n-ik)/(n-1),$ $i=1,2,$ respectively. But for $i = 1,2$ $f^\prime(c_i) = - 1/c_i,$ $U_{(n-ik)} \to 1$ in probability and $(n-ik)/(n-1)\to 1,$ thus $f^\prime(c_i) \to - 1$ in probability. This and the Slutsky theorem imply \eqref{2exp}.
\end{proof}

The schemes of the proofs of Theorems \ref{T1}
 and \ref{T3} coincide. However, the proof of the latter is more difficult technically since the statistic $\widehat{R}_{k,n}$ depends on $X_{(n-k)}$ and $X_{(n-2k)}$ together. Moreover, if $\gamma>0$ then $\sqrt{k}(\widehat{R}_{k,n} - R_{k,n})$ no longer tends to $0$ in probability.
\\
\\
{\bf Explicit form of $\mathbb{E}(\widehat R_{k,n} - R_{k,n}\mid X_{(n-k)} = q, X_{(n-2k)} = \hat q).$}\\
\\
Let us write $\sqrt{k}(\widehat{R}_{k,n} - R_{k,n})$ in explicit form
\begin{align*}
 \sqrt{k}(\widehat{R}_{k,n} - R_{k,n}) =
 \sqrt{k}\left(\ln\frac{k}{n} - \ln(1 - F_0(X_{(n-k)})) \right) \hspace{5cm}\\
- \frac{1}{\sqrt{k}}\sum_{i=n-k+1}^{n}\left[\ln\left(1-F_0\Big(u_0 + \frac{X_{(i)} - X_{(n-k)}}{X_{(n-k)} - X_{(n-2k)}}(u_0 - \hat u_0)\Big)\right)- \ln(1
- F_0(X_{(i)}))\right],
\end{align*}
where $\hat u_0 = u_0(n/(2k)),$ and consider its conditional distribution given $X_{(n-k)} = q$ and $X_{(n-2k)} = \hat q.$ Using argument similar to that from the proof of Theorem \ref{T1}
 and inheriting the notation from there, we get
\begin{eqnarray} \nonumber
\sqrt{k}(\widehat{R}_{k,n} - R_{k,n}) \mid X_{(n-k)} = q, X_{(n-2k)} = \hat q  \hspace{6cm} \\ \label{initial}
\stackrel{d}{=}
 \sqrt{k}\left(\ln\frac{k}{n} - \ln(1 - F_0(q)) \right) \hspace{7cm} \\ \nonumber
-\, \frac{1}{\sqrt{k}}\sum_{i=1}^{k}\left[\ln\left(1-F_0\Big(u_0 + \frac{X_{i}^\ast - q}{q - \hat q}(u_0 - \hat u_0)\Big)\right)- \ln(1
- F_0(X_{i}^\ast))\right].
\end{eqnarray}
Hereinafter we write $\overline{F}(x)$ instead of $1 - F(x).$ Write
\[Y^\ast = \ln\overline{F}_0\Big(u_0 + \frac{X_{1}^\ast - q}{q - \hat q}(u_0 - \hat u_0)\Big)- \ln \overline{F}_0(X_{1}^\ast)\] and find its mean. After integrating by parts, we obtain
\begin{eqnarray*}&&\mathbb{E} Y^\ast = \ln \frac{\overline{F}_0(u_0)}{\overline{F}_0(q)}\\&& - \int\limits_{q}^\infty \frac{\overline{F}_0(x)}{\overline{F}_0(q)}
\left(\frac{u_0 - \hat u_0}{q - \hat q}\frac{p_0\big(u_0 + (u_0 - \hat u_0)(x-q)/(q - \hat q)\big)}{\overline{F}_0\big(u_0 + (u_0 - \hat u_0)(x-q)/(q - \hat q)\big)} - \frac{p_0(x)}{\overline{F}_0(x)}\right) dx\\
&& =: \ln \frac{\overline{F}_0(u_0)}{\overline{F}_0(q)} + 1 - f(q, \hat q).\end{eqnarray*}
from which and \eqref{initial} we get
\[\mathbb{E}(\widehat{R}_{k,n} -
R_{k,n} \mid X_{(n-k)} = q, X_{(n-2k)} = \hat q) = f(q, \hat q) - 1.\]
\\
{\bf Evaluating the asymptotics of $\sqrt{k}\mathbb{E}(\tilde R_{k,n} - R_{k,n} \mid X_{(n-k)}, X_{(n-2k)}).$}\\
\\
Now find the asymptotics of \[\sqrt{k} (f(X_{(n-k)}, X_{(n-2k)}) - 1) = \sqrt{k}\mathbb{E}(\tilde R_{k,n} - R_{k,n}\mid X_{(n-k)}, X_{(n-2k)}).\] In contrast to the proof of Theorem \ref{T1},
 we need the multivariate delta method. First of all, observe that
\[(X_{(n-k)}, X_{(n-2k)}) \stackrel{d}{=} \left(u_0\Big(\frac{1}{1 - \exp(-E_{(k)})}\Big), u_0\Big(\frac{1}{1 - \exp(-E_{(2k)})}\Big)\right),\]
where the order statistics $E_{(k)}$ and $E_{(2k)}$ were introduced in Lemma \ref{L2}. Applying the multivariate delta method for the function
\[g(x,y) = f\left(u_0\Big(\frac{1}{1 - e^{-x}}\Big), u_0\Big(\frac{1}{1 - e^{-y}}\Big)\right)\] to the relation \eqref{2exp}, we get
\begin{equation}\label{delta2dim}\frac{\sqrt{k}\big(f(X_{(n-k)}, X_{(n-2k)}) - f(u_0, \hat u_0)\big)}{\frac{k}{n}\sqrt{D}} \xrightarrow{d} N(0,1),\end{equation}
where
\[D = \nabla g \left(\begin{array}{cc} 1 & 1 \\ 1 & 2 \end{array}\right) (\nabla g)^T\] and $\nabla g$ is the gradient of $g(x,y)$ at the point $(-\ln(1 - k/n), -\ln(1 - 2k/n)).$ The correctness of applying the multivariate delta method here is established using the same argumentation as of \eqref{delta1} (in fact, \eqref{delta2dim} can be derived from the two-dimensional version of \eqref{delta}), is very technical and thus omitted. 

Observe that $f(u_0, \hat u_0) = 1.$ Thus the asymptotics of $\sqrt{k}\mathbb{E}(\tilde R_{k,n} - R_{k,n} \mid X_{(n-k)}, X_{(n-2k)})$ strongly depend on those of $\frac{k}{n}\sqrt{D}.$ One can show that
\begin{eqnarray}\frac{k}{n}\sqrt{D} \to \frac{1}{\sqrt{2}(2^\gamma -1)}\frac{\gamma}{\gamma+1}, \label{notgreat} \end{eqnarray}
where for $\gamma = 0$ the right-hand side should be understood as $(\sqrt{2}\ln 2)^{-1}.$ The proof of \eqref{notgreat} is technical and includes multiple use of the properties of regularly-varying functions and L'H$\hat{\text{o}}$pital's rule; it is given in the Supplementary material. \\
\\
Summarizing the above, we proved that under the assumptions of the theorem, \begin{equation}
\sqrt{k}\big(f(X_{(n-k)}, X_{(n-2k)})  - 1\big) \stackrel{d}{\rightarrow} N\Big(0,
1 + \frac{\gamma^2}{2(\gamma+1)^2(2^\gamma-1)^2}\Big),
\label{strange}\end{equation}
where for $\gamma=0$ the ratio on the right-hand side should be understood as $1/(2 (\ln 2)^2).$
\\
\\
{\bf Final steps}
\\
\\
The proof of the relation
\begin{equation}\label{2P}\sqrt{k}\big(\widehat{R}_{k,n} - R_{k,n}\big) - \sqrt{k}\big(f(X_{(n-k)}, X_{(n-2k)})  - 1\big) \xrightarrow{P} 0\end{equation} does not contain new ideas as compared to the proof of the similar relation \eqref{P} and the previous steps of the current proof, so we omit it. Finally, we have
\begin{eqnarray*}\sqrt{k}(\widehat{R}_{k,n} - 1) &=& \sqrt{k}\big({R}_{k,n} - 1) \\
&& +\, \sqrt{k}\Big[\big(\widehat{R}_{k,n} - R_{k,n}\big) - \big(f(X_{(n-k)}, X_{(n-2k)})  - 1\big)\Big]\\
&& +\, \sqrt{k}\big(f(X_{(n-k)}, X_{(n-2k)})  - 1\big).
\end{eqnarray*}
Observe that the properties of iid exponential random variables imply that the third summand on the right-hand side of the latter relation is independent of the first. Indeed, recall that by \eqref{rknexp}
\[{R}_{k,n} \stackrel{d}{=} \frac{1}{k} \sum_{i=0}^{k-1} (E_{(n-i)} - E_{(n-k)})\]
and $-\ln (1 - F_0(X_{(i)})) \stackrel{d}{=} E_{(i)},$ $i = 1, \ldots, n,$ with $\{E_i\}_{i=1}^n$ i.i.d. standard exponential.
 Next, it follows from the R\'enyi representation \cite{renyi} that $E_{(n-k)}$ and $E_{(n-2k)}$ are independent of $\{E_{(n-i)} - E_{(n-k)}\}_{i=0}^{k-1}.$ Hence, $\ln (1 - F_0(X_{(n-k)}))$ and $\ln (1 - F_0(X_{(n-2k)}))$ are independent of $R_{k,n},$ and the required statement follows from the fact that \[(X_{(n-k)},X_{(n-2k)}) = \Big(g(-\ln (1 - F_0(X_{(n-k)}))), g(-\ln (1 - F_0(X_{(n-2k)})))\Big)\] a.s. with $g(x) = u_0\big((1 - e^{-x})^{-1}\big).$
This argument together with \eqref{old}, \eqref{strange} and \eqref{2P} completes the proof of Theorem \ref{T3}.

\subsection{Proof of Theorem \ref{T4}
}

The scheme of the proof of Theorem \ref{T4}
 is the same as for Theorem \ref{T2},
 so it is omitted. Note only, that for the proof of the relation similar to \eqref{appr2}, one needs to use the multivariate delta method and Lemma \ref{L2}. Next, instead of the relation \eqref{ineq} used to find the upper bound for $P(Y_1\le x)$, one should use the relation
\[\frac{\overline{F}_1(u_1(t) + x(u_1(t) - u_1(t/2)))}{\overline{F}_1(u_1(t))} \ge \frac{\big(\overline{F}_0(u_0(t) + x(u_0(t) - u_0(t/2)))\big)^{1-\varepsilon}}{(\overline{F}_0(u_0(t)))^{1-\varepsilon}}\] for all $x>0.$ The latter follows from the condition $C_\delta(F_0, F_1).$
Indeed,
\begin{eqnarray*} \frac{\overline{F}_1(u_1(t) + x(u_1(t) - u_1(t/2)))}{\overline{F}_1(u_1(t))} &=& \frac{\overline{F}_1\big(u_1(t)(1 + x - xu_1(t/2)/u_1(t))\big)}{\overline{F}_1(u_1(t))} \\
&\ge& \frac{\overline{F}_1\big(u_1(t)(1 + x - xu_0(t/2)/u_0(t))\big)}{\overline{F}_1(u_1(t))}\\ &\ge& \frac{\big(\overline{F}_0(u_0(t) + x(u_0(t) - u_0(t/2)))\big)^{1-\varepsilon}}{(\overline{F}_0(u_0(t)))^{1-\varepsilon}},
\end{eqnarray*}
where the first and second relations follow from $C_0(F_0, F_1)$ and Proposition 1, respectively.

\subsection{Proof of Corollary \ref{C2}
}
The proof of Corollary \ref{C2}
 is like that of Corollary \ref{C1}
  except that instead of the relation \eqref{v} one should use the inequality
\[\frac{V_{(i)} - V_{(n-k)}}{V_{(n-k)} - V_{(n-2k)}} \le \frac{X_{(i)} - X_{(n-k)}}{X_{(n-k)} - X_{(n-2k)}},
\] that immediately follows from $C_0(F_0, F_1)$ and \eqref{v}.

\backmatter



\section*{Declarations}
\textbf{Acknowledgements.}\\
The author wants to sincerely thank Anthony C. Davison for his valuable comments and advice, which made the manuscript much better.
\\
\\
\textbf{Conflict of interest statement.}\\
The author has no conflict of interest to declare.\\
\\
\textbf{Data availability statement.}\\ The datasets analysed during the current study are available on the site of Daily Global Historical Climatology Network, https://www.ncei.noaa.gov/data/global-historical-climatology-network-daily/ .




\newpage
\setcounter{page}{1}
\pagenumbering{roman}
\thispagestyle{empty}
\medskip
\centerline{\large Supplementary Material:}
\medskip
\centerline{\large Location- and scale-free procedures for
distinguishing}\smallskip \centerline{\large {between distribution tail
models}}
\medskip
\centerline{\normalsize by Igor Rodionov}

\subsection{{ The proof of the relation \eqref{notgreat}}} \label{appendix}

{\bf Evaluating the asymptotic of $\frac{k}{n}\sqrt{D}$ for $\gamma>0.$}\\
\\
 First, observe that
\[\nabla g = \left(f^\prime_x(u_0, \hat u_0) u_0^\prime(n/k) \frac{1 - k/n}{(k/n)^2}, \, f^\prime_y(u_0, \hat u_0) u_0^\prime(n/(2k)) \frac{1 - 2k/n}{(2k/n)^2} \right).\] Next, it can be proved that
\begin{eqnarray*}f^\prime_x(u_0, \hat u_0) = \frac{n}{k}\left(p_0(u_0) - \int\limits_{u_0}^\infty\frac{x - \hat u_0}{u_0 - \hat u_0}\frac{p_0^2(x)}{\overline{F}_0(x)}dx\right),
\\ f^\prime_y(u_0, \hat u_0) = \frac{n}{k}\int\limits_{u_0}^\infty\frac{x - u_0}{u_0 - \hat u_0}\frac{p_0^2(x)}{\overline{F}_0(x)}dx.\end{eqnarray*} Now assume $\gamma>0.$ Using the relation \eqref{lopital} established above and the equality
\begin{equation} \label{lopital2}
\lim_{u\to\infty}\frac{p_0(u)}{\int_q^\infty\frac{p_0^2(x)}{\overline{F}_0(x)}dx} = \gamma + 1,\end{equation} that can be obtained similarly, we derive that
\begin{eqnarray*}f^\prime_x(u_0, \hat u_0) = -\frac{n}{k}\frac{\hat u_0}{u_0 - \hat u_0}\frac{\gamma}{\gamma+1} p_0(u_0)(1 + o(1)), \\ f^\prime_y(u_0, \hat u_0) = \frac{n}{k}\frac{u_0}{u_0 - \hat u_0}\frac{\gamma}{\gamma+1} p_0(u_0)(1 + o(1)).\end{eqnarray*} Taking into account $p_0(u_0(t)) u_0^\prime(t) = 1/t^2,$ for $\gamma>0$ we have
\begin{eqnarray*}\frac{k}{n}\sqrt{D} &=& \frac{u_0}{u_0 - \hat u_0} \frac{\gamma}{\gamma+1}\sqrt{\Big(\frac{\hat u_0}{u_0}\Big)^2 - \frac{1}{2}\frac{\hat u_0}{u_0} \frac{u_0^\prime(n/(2k))}{u_0^\prime(n/k)} + \frac{1}{8}\Big(\frac{u_0^\prime(n/(2k))}{u_0^\prime(n/k)}\Big)^2}(1 + o(1))\\ &\to& \frac{1}{\sqrt{2}(2^\gamma -1)}\frac{\gamma}{\gamma+1}\end{eqnarray*} by properties of regularly varying functions (namely, we use here the relations (1.1.33) and (1.2.18), \cite{dehaan}).
\\
\\
{\bf Evaluating the asymptotic of $\frac{k}{n}\sqrt{D}$ for $\gamma=0.$}\\
\\
Consideration of the case $\gamma = 0$ is a bit more delicate. First, note that
\[f^\prime_x(u_0, \hat u_0) + f^\prime_y(u_0, \hat u_0) = \frac{n}{k}\left(p_0(u_0) - \int\limits_{u_0}^\infty\frac{p_0^2(x)}{\overline{F}_0(x)}dx\right) = o\Big(\frac{n}{k} p_0(u_0)\Big)\]
by \eqref{lopital2}. Next, using the latter and again the relations $p_0(u_0(t)) u_0^\prime(t) = 1/t^2$ and (1.1.33), \cite{dehaan}, we derive
\[\frac{k}{n}\nabla g = \frac{f^\prime_y(u_0, \hat u_0)}{p_0(u_0)\cdot n/k}\big(-1 + o(1), \, 1/2 + o(1)\big),\] thus
\[\frac{k}{n}\sqrt{D} = \frac{k}{n}\frac{f^\prime_y(u_0, \hat u_0)}{\sqrt{2}p_0(u_0)}(1 + o(1)).\] Find the asymptotics of the right-hand side of the latter equality. By Corollary 1.1.10, \cite{dehaan},
\[\frac{u_0 - \hat u_0}{n/k\cdot u^\prime_0(n/k)} \to \ln 2,\] thus, using $p_0(u_0) u_0^\prime(n/k) = (k/n)^2$ and $\overline{F_0}(u_0) = k/n,$ we have
\begin{eqnarray*}\frac{k}{n}\frac{f^\prime_y(u_0, \hat u_0)}{p_0(u_0)} &=& \frac{1}{p_0(u_0)} \int\limits_{u_0}^\infty\frac{x - u_0}{u_0 - \hat u_0}\frac{p_0^2(x)}{\overline{F}_0(x)}dx\\ &=& \frac{n/k\cdot u^\prime_0(n/k)}{u_0 - \hat u_0} \cdot \frac{1}{\overline{F_0}(u_0)} \int\limits_{u_0}^\infty(x - u_0)\frac{p_0^2(x)}{\overline{F}_0(x)}dx.\end{eqnarray*} The first multiplier on the right-hand side of the latter relation tends to $(\ln 2)^{-1},$ hence, it remains to find the asymptotics of the second one. By the L'H$\hat{\text{o}}$pital's rule, we have
\begin{eqnarray*}&& \lim_{u\to\infty}\frac{1}{\overline{F_0}(u)} \int\limits_{u}^\infty(x - u)\frac{p_0^2(x)}{\overline{F}_0(x)}dx \\&& =  \lim_{u\to\infty} \frac{1}{- p_0(u)} \left(-\frac{u p^2(u)}{\overline{F}(u)} + \frac{u p_0^2(u)}{\overline{F}_0(u)} - \int_{u}^\infty \frac{p^2_0(x)}{\overline{F}_0(x)} dx \right) \\ && =  \lim_{u\to\infty} \frac{1}{p_0(u)} \int_{u}^\infty \frac{p^2_0(x)}{\overline{F}_0(x)} dx  = 1,\end{eqnarray*} where the latter equality follows from \eqref{lopital2}. Thus, we show that if $\gamma = 0$ then $k \sqrt{D}/n \to (\sqrt{2} \ln 2)^{-1}.$

\end{document}